\documentclass[11pt]{article}
\usepackage{amssymb,amsmath,amsfonts,amsthm,mathrsfs}
\usepackage{graphicx,graphics,psfrag,epsfig,calrsfs,cancel}
\usepackage[colorlinks=true, pdfstartview=FitV, linkcolor=blue, citecolor=red, urlcolor=blue]{hyperref}

\usepackage{caption,subfig,color}
\usepackage[dvipsnames]{xcolor}
\usepackage{tikz}
\usetikzlibrary{calc}

\usepackage{geometry}
\newtheorem{theorem}{Theorem}[section]

\newtheorem{proposition}[theorem]{Proposition}

\newtheorem{lemma}[theorem]{Lemma}
\newtheorem{corollary}[theorem]{Corollary}

\numberwithin{equation}{section}
\numberwithin{figure}{section}

\newcommand{\N}{\mathbb{N}}
\newcommand{\Z}{\mathbb{Z}}
\newcommand{\R}{\mathbb{R}}

\begin{document}

\title{The Neumann boundary condition \\
for the two-dimensional Lax-Wendroff scheme}

\author{Antoine {\sc Benoit}\thanks{Email: {\tt antoine.benoit@univ-littoral.fr}. Research of the author was supported by 
ANR project NABUCO, ANR-17-CE40-0025.} $\,$ \& Jean-Fran\c{c}ois {\sc Coulombel}\thanks{Institut de Math\'ematiques de Toulouse ; 
UMR5219, Universit\'e de Toulouse ; CNRS, F-31062 Toulouse Cedex 9, France. Email: 
{\tt jean-francois.coulombel@math.univ-toulouse.fr}. Research of the author was supported by ANR project NABUCO, ANR-17-CE40-0025.}}
\date{\today}
\maketitle

\begin{abstract}
We study the stability of the two-dimensional Lax-Wendroff scheme with a stabilizer that approximates solutions to the transport equation. 
The problem is first analyzed in the whole space in order to show that the so-called energy method yields an optimal stability criterion for 
this finite difference scheme. We then deal with the case of a half-space when the transport operator is outgoing. At the numerical level, 
we enforce the Neumann extrapolation boundary condition and show that the corresponding scheme is stable. Eventually we analyze the 
case of a quarter-space when the transport operator is outgoing with respect to both sides. We then enforce the Neumann extrapolation 
boundary condition on each side of the boundary and propose an extrapolation boundary condition at the numerical corner in order to 
maintain stability for the whole numerical scheme.
\end{abstract}



\paragraph{Notation.} For $d$ a positive integer and $\mathcal{J} \subset \Z^d$, we let $\ell^2(\mathcal{J};\R)$ denote the Hilbert space of 
real valued, square integrable sequences indexed by $\mathcal{J}$ and equipped with the norm:
$$
\forall \, u \in \ell^2(\mathcal{J};\R) \, ,\quad \| \, u \, \|_{\ell^2(\mathcal{J})}^2 \, := \, \sum_{j \in \mathcal{J}} \, u_j^2 \, .
$$
The corresponding scalar product is denoted $\langle \, \, ; \, \, \rangle_{\ell^2(\mathcal{J})}$.

\section{Introduction}

We explore in this article the relevance of extrapolation boundary conditions for outgoing transport equations in \emph{two} space dimensions. 
In one space dimension, a general stability and convergence theory has been developed in \cite{jfcfl} following, among others, previous works 
by Kreiss and Goldberg \cite{kreissproc,goldberg}. The main results in \cite{kreissproc,goldberg,jfcfl} assert that, for an \emph{explicit}, 
\emph{one time step} finite difference scheme that is stable in $\ell^2(\Z;\R)$ and that is consistent with the transport equation:
$$
\partial_t \, v \, + \, a \, \partial_x \, v \, = \, 0 \, ,\quad a \, < \, 0 \, ,\quad (t,x) \in \R^+ \times \R \, ,
$$
then extrapolation numerical boundary conditions at the origin yield a numerical scheme that is stable in $\ell^2(\N;\R)$ (that is, on the half-line). 
The result is independent of the extrapolation order that is chosen at the boundary. We refer to \cite[Theorem 3.1]{jfcfl} for a detailed statement. 
We aim here at understanding the influence of \emph{tangential} directions on this stability result in higher space dimension.

Numerical boundary conditions for two-dimensional hyperbolic problems have been investigated, for instance, in \cite{ag2,sloan}, by the 
so-called normal mode analysis. This method is, to some extent, optimal to characterize stability but it usually leads to rather involved 
algebraic calculations that, sometimes, cannot be carried out. In this article, we rather wish to developed an \emph{energy method} in order 
to deal with more involved geometries as the quarter-space. We focus on a second-order discretization of the transport equation that was 
originally proposed by Lax and Wendroff \cite{laxwendroff}. This approximation does not rely on any dimensional splitting, which prevents 
us from using one-dimensional arguments, and it is second order accurate with a compact (nine point) stencil, which is a good compromise 
between efficiency and complexity. We aim at exploring finite difference schemes with wider stencils and develop a general stability theory 
in the future.

The plan of the article is as follows. Section \ref{section2} is devoted to the definition of the finite difference scheme in the whole space and to 
the stability analysis by means of the energy method without any boundary. Previous stability results for this numerical scheme, such as in the 
references \cite{laxwendroff,tadmor,jfc}, were relying on the Fourier transform, which is not convenient for half-space or quarter-space problems. 
In Section \ref{section2}, we recover the \emph{optimal} stability criterion for the Lax-Wendroff scheme \eqref{LW} below by only using elementary 
energy arguments (discrete integration by parts and Cauchy-Schwarz inequalities). As far as we know, even this part of our analysis is new. We 
then apply a similar strategy in Section \ref{section3} to deal with half-space problems. Our main result in this section is Theorem \ref{thm1} in 
which we prove that the first order extrapolation boundary condition (see \eqref{extrapolation} below) maintains stability for the corresponding 
numerical scheme in a half-space. We even recover a trace estimate for the solution, which is in agreement with the fulfillment of the Uniform 
Kreiss-Lopatinskii condition (see, e.g., \cite{gks,gko,michelson}). Eventually, we consider in Section \ref{section4} the quarter-space with a transport 
operator that is outgoing with respect to both sides of the boundary. In view of Theorem \ref{thm1}, we may expect that enforcing an extrapolation 
numerical boundary condition on each side is a good starting point for deriving a stable scheme. However, the extrapolation procedure on each 
side of the boundary still leaves one undetermined quantity at each time step, which is the value of the numerical solution at the \emph{corner} 
of the space domain. Applying our energy argument, we are able to propose an extrapolation numerical corner condition that maintains stability. 
Numerical evidence suggests that ``wrong'' corner conditions may yield a strongly unstable scheme.

\section{Stability for the Cauchy problem}
\label{section2}

\subsection{Definition of the numerical scheme}

We consider the two-dimensional transport equation on the whole space $\R^2$:
\begin{equation}
\label{hcl}
\begin{cases}
\partial_t u \, + \, a \, \partial_x u \, + \, b \, \partial_y u \, = \, 0 \, ,& t \ge 0 \, ,\, (x,y) \in \R^2 \, ,\\
u_{|_{t=0}} \, = \, u_0 \, ,& 
\end{cases}
\end{equation}
where $a,b$ are some given real numbers. We make no sign assumption on $a,b$ in this section. The initial condition $u_0$ in \eqref{hcl} belongs 
to the Lebesgue space $L^2(\R^2;\R)$. We consider below a finite difference approximation of \eqref{hcl} that is defined as follows. Given some 
space steps $\Delta x, \Delta y >0$ in each spatial direction, and given a time step $\Delta t>0$, we introduce the ratios $\lambda := \Delta t/\Delta x$ 
and $\mu :=\Delta t/\Delta y$. In all what follows, the ratios $\lambda$ and $\mu$ are assumed to be fixed, meaning that they are given a priori of the 
computation and are meant to be tuned in order to satisfy some stability requirements (the so-called Courant-Friedrichs-Lewy condition \cite{cfl}, later 
on referred to as the CFL condition, see for instance Corollary \ref{cor1} below). The solution $u$ to \eqref{hcl} is then approximated on the time-space 
domain $[n \, \Delta t,(n+1) \, \Delta t) \times [(j-1/2) \, \Delta x,(j+1/2) \, \Delta x) \times [(k-1/2) \, \Delta y,(k+1/2) \, \Delta y)$ by a real number $u_{j,k}^n$ 
for any $n \in \N$ and $(j,k) \in \Z^2$. The discrete initial condition $u^0$ is defined for instance by taking the piecewise constant projection of $u_0$ in 
\eqref{hcl} on each cell, that is (see \cite{gko}):
$$
\forall \, (j,k) \in \Z^2 \, ,\quad u^0_{j,k} \, := \, \dfrac{1}{\Delta x \, \Delta y} \, \int_{(j-1/2) \, \Delta x}^{(j+1/2) \, \Delta x} \, 
\int_{(k-1/2) \, \Delta y}^{(k+1/2) \, \Delta y} \, u_0(x,y) \, {\rm d}x \, {\rm d}y \, .
$$
This initial condition satisfies:
$$
\Delta x \, \Delta y \, \| \, u^0 \, \|_{\ell^2(\Z^2)}^2 \, = \, \sum_{(j,k) \in \Z^2} \, \Delta x \, \Delta y \, (u^0_{j,k})^2 \, \le \, \| \, u_0 \, \|_{L^2(\R^2)}^2 \, .
$$
It then remains to determine the $u_{j,k}^n$'s inductively with respect to $n$. The Lax-Wendroff scheme with a stabilizer reads (see \cite{laxwendroff}):
\begin{align}
u_{j,k}^{n+1} \, = \, u_{j,k}^n 
& - \, \dfrac{\lambda \, a}{2} \, \big( u_{j+1,k}^n-u_{j-1,k}^n \big) \, - \, \dfrac{\mu \, b}{2} \, \big( u_{j,k+1}^n-u_{j,k-1}^n \big) \notag \\
& + \, \dfrac{(\lambda \, a)^2}{2} \, \big( u_{j+1,k}^n-2\, u_{j,k}^n+u_{j-1,k}^n \big) \, + \, \dfrac{(\mu \, b)^2}{2} \, \big( u_{j,k+1}^n-2\, u_{j,k}^n+u_{j,k-1}^n \big) \notag \\
& + \, \dfrac{\lambda \, a \, \mu \, b}{4} \, \big( u_{j+1,k+1}^n-u_{j+1,k-1}^n-u_{j-1,k+1}^n+u_{j-1,k-1}^n \big) \label{LW} \\
& - \, \dfrac{(\lambda \, a)^2 + (\mu \, b)^2}{8} \, \big( u_{j+1,k+1}^n-2\, u_{j+1,k}^n+u_{j+1,k-1}^n \notag \\
& \qquad \qquad \qquad \qquad - \, 2\, u_{j,k+1}^n+4\, u_{j,k}^n-2\, u_{j,k-1}^n+u_{j-1,k+1}^n-2\, u_{j-1,k}^n+u_{j-1,k-1}^n \big) \, ,\notag
\end{align}
where $(j,k)$ belongs to $\Z^2$. The so-called stabilizing term corresponds to the last two lines on the right hand side of \eqref{LW}. This term is meant to 
add some (rather weak) dissipation that improves the stability properties of the finite difference approximation. We refer to \cite{laxwendroff,gko} for alternative 
approximations of \eqref{hcl}.

We first recall some stability results for the numerical scheme \eqref{LW} and then propose an energy method in order to recover the optimal stability 
criterion for \eqref{LW}. The relevance of the energy method is made more precise below in Sections \ref{section3} and \ref{section4} when we extend 
our approach to more involved geometries.

\subsection{A reminder on the Fourier approach}

The $\ell^2$ stability of the iteration \eqref{LW} was first analyzed in \cite{laxwendroff} (for symmetric hyperbolic systems) by means of the numerical 
radius of the amplification matrix. When one specifies the result of \cite{laxwendroff} to the scalar case, the main result of \cite{laxwendroff} shows 
that \eqref{LW} is stable in $\ell^2(\Z^2;\R)$, that is:
$$
\forall \, n \in \N \, ,\quad \| \, u^{n+1} \, \|_{\ell^2(\Z^2)} \, \le \, \| \, u^n \, \|_{\ell^2(\Z^2)} \, ,
$$
\emph{if, and only if}, the parameters $\lambda,\mu$ satisfy the restriction:
\begin{equation}
\label{CFLlaxwendroff}
(\lambda \, a)^2 \, + \, (\mu \, b)^2 \, \le \, \dfrac{1}{2} \, .
\end{equation}
The extension of this result to symmetric hyperbolic \emph{systems} is the purpose of \cite{laxwendroff,tadmor, jfc}. However, all these references are 
based on Fourier analysis and a sharp estimate of the numerical radius or of the norm of the amplification matrix. This technique is of little use for 
half-space or quarter-space problems as we intend to study below. We thus propose below an alternative proof of the stability result of \cite{laxwendroff} 
for \eqref{LW} by using the energy method. We restrict for simplicity to the scalar case since our main concern is to deal with extrapolation procedures 
for outgoing transport equations. Our goal is to recover the same sufficient condition \eqref{CFLlaxwendroff} for stability of \eqref{LW} in $\ell^2(\Z^2;\R)$ 
(the necessity of \eqref{CFLlaxwendroff} for stability is proved in \cite{laxwendroff} by computing the amplification matrix at the frequency $(\pi,\pi)$).

\subsection{The energy method}

We now explain how the energy method gives the optimal stability criterion \eqref{CFLlaxwendroff} for \eqref{LW} on the whole space $\Z^2$. Our main 
result is Corollary \ref{cor1} at the end of this section. Since the result is not new, we rather focus on the method and all intermediate steps and postpone 
the statement of the main result to the end once all preliminary steps have been achieved. We thus start from the definition \eqref{LW} and decompose 
$u_{j,k}^{n+1}$ into three pieces:
$$
\forall \, (j,k) \in \Z^2 \, ,\quad u_{j,k}^{n+1} \, = \, u_{j,k}^n \, - \, w_{j,k}^n \, + \, v_{j,k}^n \, ,
$$
where $v_{j,k}^n$ and $w_{j,k}^n$ are defined by:
\begin{subequations}
\label{defvwn}
\begin{align}
v_{j,k}^n \, :=& \, - \, \dfrac{\lambda \, a}{2} \, \big( u_{j+1,k}^n-u_{j-1,k}^n \big) \, - \, \dfrac{\mu \, b}{2} \, \big( u_{j,k+1}^n-u_{j,k-1}^n \big) \, ,\label{defvjkn} \\
w_{j,k}^n \, :=& \, - \, \dfrac{(\lambda \, a)^2}{2} \, \big( u_{j+1,k}^n-2\, u_{j,k}^n+u_{j-1,k}^n \big) \, - \, 
\dfrac{(\mu \, b)^2}{2} \, \big( u_{j,k+1}^n-2\, u_{j,k}^n+u_{j,k-1}^n \big) \notag \\
& - \, \dfrac{\lambda \, \mu \, a \, b}{4} \, \big( u_{j+1,k+1}^n-u_{j+1,k-1}^n-u_{j-1,k+1}^n+u_{j-1,k-1}^n \big) \label{defwjkn} \\
& + \, \dfrac{(\lambda \, a)^2 + (\mu \, b)^2}{8} \, \big( u_{j+1,k+1}^n-2\, u_{j+1,k}^n+u_{j+1,k-1}^n \notag \\
& \qquad \qquad \qquad \qquad - \, 2 \, u_{j,k+1}^n+4\, u_{j,k}^n-2\, u_{j,k-1}^n+u_{j-1,k+1}^n-2\, u_{j-1,k}^n+u_{j-1,k-1}^n \big) \, ,\notag
\end{align}
\end{subequations}

It will be useful below to use operator notations in order to highlight symmetry or skew-symmetry properties. We thus introduce the following discrete first 
order partial derivatives and Laplacians:
\begin{align*}
&(D_{1,+}U)_{j,k} \, := \, U_{j+1,k} \, - \, U_{j,k} \, ,\quad (D_{1,-}U)_{j,k} \, := \, U_{j,k} \, - \, U_{j-1,k} \, ,\\
&(D_{2,+}U)_{j,k} \, := \, U_{j,k+1} \, - \, U_{j,k} \, ,\quad (D_{2,-}U)_{j,k} \, := \, U_{j,k} \, - \, U_{j,k-1} \, ,\\
&D_{1,0} \, := \, \dfrac{D_{1,+} +D_{1,-}}{2} \, ,\quad D_{2,0} \, := \, \dfrac{D_{2,+} +D_{2,-}}{2} \, ,\quad 
\Delta_1 \, := \, D_{1,+} \, D_{1,-} \, ,\quad \Delta_2 \, := \, D_{2,+} \, D_{2,-} \, .
\end{align*}
In order to keep the notation as simple as possible, we write below $D_{1,+}u_{j,k}$ rather than $(D_{1,+}u)_{j,k}$ and analogously for other operators. 
We hope that this does not create any confusion. With such definitions, the operators $D_{1,0}$ and $D_{2,0}$ are skew-selfadjoint on $\ell^2(\Z^2;\R)$ 
and the operators $\Delta_1$, $\Delta_2$ are selfadjoint on $\ell^2(\Z^2;\R)$, see \cite{gko} (this is known as \emph{discrete integration by parts} or 
\emph{Abel's transform}). We also have $D_{1,-}^*=-D_{1,+}$ and $D_{2,-}^*=-D_{2,+}$ for the $\langle \, \cdot \, ; \cdot \rangle_{\ell^2(\Z^2)}$ scalar 
product. Let us eventually observe that all operators defined above commute, which will also be useful below.

The above definitions allow us to rewrite \eqref{defvwn} in a compact form as:
\begin{subequations}
\label{defvwn'}
\begin{align}
v^n \, :=& \, - \, \lambda \, a \, D_{1,0} \, u^n \, - \, \mu \, b \, D_{2,0} \, u^n \, ,\label{defvjkn'} \\
w^n \, :=& \, - \, 
\dfrac{(\lambda \, a)^2}{2} \, \Delta_1 \, u^n \, - \, \dfrac{(\mu \, b)^2}{2} \, \Delta_2 \, u^n \, - \, \lambda \, \mu \, a \, b \, D_{1,0} \, D_{2,0} \, u^n 
\, + \, \dfrac{(\lambda \, a)^2 + (\mu \, b)^2}{8} \, \Delta_1 \, \Delta_2 \, u^n \, .\label{defwjkn'}
\end{align}
\end{subequations}
As a consequence, we observe that $v^n$ gathers all the skew-selfadjoint operators acting on $u^n$ and $w^n$ gathers all the selfadjoint 
operators acting on $u^n$. In particular, we easily obtain the following relation:
\begin{equation}
\label{energie-cauchy}
\| \, u^{n+1} \, \|_{\ell^2(\Z^2)}^2 \, - \, \| \, u^n \, \|_{\ell^2(\Z^2)}^2 \, = \, \| \, w^n \, \|_{\ell^2(\Z^2)}^2 \, + \, \| \, v^n \, \|_{\ell^2(\Z^2)}^2 \, - \, 
2 \, \langle u^n \, ; \, w^n \rangle_{\ell^2(\Z^2)} \, ,
\end{equation}
because $v^n$ is orthogonal to both $u^n$ and $w^n$ in $\ell^2(\Z^2;\R)$. The stability analysis of the Lax-Wendroff scheme \eqref{LW} in the 
whole space $\Z^2$ then relies on the following two results, whose proof will be given below.

\begin{lemma}
\label{lem1}
Let $u^n \in \ell^2(\Z^2;\R)$, and let the sequences $v^n,w^n \in \ell^2(\Z^2;\R)$ be defined by \eqref{defvwn'}. Then there holds:
\begin{align*}
\| \, v^n \, \|_{\ell^2(\Z^2)}^2 \, - \, 2 \, \langle u^n \, ; \, w^n \rangle_{\ell^2(\Z^2)} \, = \, & 
\, - \, \dfrac{(\lambda \, a)^2}{4} \, \| \, \Delta_1 \, u^n \, \|_{\ell^2(\Z^2)}^2 \, - \, \dfrac{(\mu \, b)^2}{4} \, \| \, \Delta_2 \, u^n \, \|_{\ell^2(\Z^2)}^2 \\
& \, - \, \dfrac{(\lambda \, a)^2 +(\mu \, b)^2}{4} \, \| \, D_{1,+} \, D_{2,+} \, u^n \, \|_{\ell^2(\Z^2)}^2 \, .
\end{align*}
\end{lemma}

\begin{proposition}
\label{prop1}
Let $u^n \in \ell^2(\Z^2;\R)$, and let the sequence $w^n \in \ell^2(\Z^2;\R)$ be defined by \eqref{defwjkn'}. Then there holds:
\begin{multline}
\label{estimprop1}
4 \, \| \, w^n \, \|_{\ell^2(\Z^2)}^2 \, \le \, 2 \, \big( (\lambda \, a)^2 +(\mu \, b)^2 \big) \, 
\Big\{ (\lambda \, a)^2 \, \| \, \Delta_1 \, u^n \, \|_{\ell^2(\Z^2)}^2 \, + \, (\mu \, b)^2 \, \| \, \Delta_2 \, u^n \, \|_{\ell^2(\Z^2)}^2 \\
\, + \, \big( (\lambda \, a)^2 +(\mu \, b)^2 \big) \, \| \, D_{1,+} \, D_{2,+} \, u^n \, \|_{\ell^2(\Z^2)}^2 \Big\} \, .
\end{multline}
\end{proposition}

The remainder of this section is devoted to the proof of Lemma \ref{lem1} and Proposition \ref{prop1}. In the end, we explain how these 
two results give a stability, and even a dissipativity, estimate for the Lax-Wendroff scheme \eqref{LW} under suitable CFL conditions. Before proving 
Lemma \ref{lem1} and Proposition \ref{prop1}, we state a first crucial lemma which will be very useful below and will also guide us in the analysis of 
the half-space and quarter-space problems.

\begin{lemma}
\label{lem2}
Let $U \in \ell^2(\Z^2;\R)$. Then there holds:
\begin{subequations}
\label{formuleslem2}
\begin{align}
\| \, D_{1,0} \, U \, \|_{\ell^2(\Z^2)}^2 \, &= \, \| \, D_{1,+} \, U \, \|_{\ell^2(\Z^2)}^2 \, - \, \dfrac{1}{4} \, \| \, \Delta_1 \, U \, \|_{\ell^2(\Z^2)}^2 \, ,\label{formulelem2-1} \\
\| \, D_{2,0} \, U \, \|_{\ell^2(\Z^2)}^2 \, &= \, \| \, D_{2,+} \, U \, \|_{\ell^2(\Z^2)}^2 \, - \, \dfrac{1}{4} \, \| \, \Delta_2 \, U \, \|_{\ell^2(\Z^2)}^2 \, ,\label{formulelem2-2} \\
\| \, D_{1,0} \, D_{2,0} \, U \, \|_{\ell^2(\Z^2)}^2 \, &= \, \| \, D_{1,+} \, D_{2,+} \, U \, \|_{\ell^2(\Z^2)}^2 \notag \\
& \qquad - \, \dfrac{1}{4} \, \| \, D_{1,+} \, \Delta_2 \, U \, \|_{\ell^2(\Z^2)}^2 \, - \, \dfrac{1}{4} \, \| \, D_{2,+} \, \Delta_1 \, U \, \|_{\ell^2(\Z^2)}^2 
\, + \, \dfrac{1}{16} \, \| \, \Delta_1 \, \Delta_2 \, U \, \|_{\ell^2(\Z^2)}^2 \, .\label{formulelem2-3}
\end{align}
\end{subequations}
\end{lemma}

\begin{proof}[Proof of Lemma \ref{lem2}]
Let us start with \eqref{formulelem2-1} and \eqref{formulelem2-2}. Given three real numbers $U_{j-1},U_j,U_{j+1}$, there holds the relation:
\begin{equation}
\label{equationalgebrique}
\dfrac{(U_{j+1} \, - \, U_{j-1})^2}{4} \, + \, \dfrac{(U_{j+1} \, - \, 2 \, U_j \, + \, U_{j-1})^2}{4} 
\, = \, \dfrac{1}{2} \, (U_{j+1} \, - \, U_j)^2 \, + \, \dfrac{1}{2} \, (U_j \, - \, U_{j-1})^2 \, .
\end{equation}
We now observe that for a square integrable sequence $U \in \ell^2(\Z;\R)$, the two terms on the right hand side of \eqref{equationalgebrique} have 
equal sum:
$$
\sum_{j \in \Z} \, (U_{j+1} \, - \, U_j)^2 \, = \, \sum_{j \in \Z} \, (U_j \, - \, U_{j-1})^2 \, .
$$
Consequently, given $U \in \ell^2(\Z;\R)$, we have the relation (with rather obvious notation for sequences indexed by $\Z$ rather than by $\Z^2$):
\begin{equation}
\label{lem2-formule1D}
\forall \, U \in \ell^2(\Z;\R) \, ,\quad \| \, D_0 \, U \, \|_{\ell^2(\Z)}^2 \, = \, \| \, D_+ \, U \, \|_{\ell^2(\Z)}^2 \, - \, \dfrac{1}{4} \, \| \, \Delta \, U \, \|_{\ell^2(\Z)}^2 \, .
\end{equation}
This formula implies both \eqref{formulelem2-1} and \eqref{formulelem2-2} by using Fubini's Theorem, that is by summing first with respect to $j$ 
or $k$.

The last formula \eqref{formulelem2-3} in Lemma \ref{lem2} is a consequence of \eqref{formulelem2-1} and \eqref{formulelem2-2} by computing 
(use \eqref{formulelem2-1}):
$$
\| \, D_{1,0} \, D_{2,0} \, U \, \|_{\ell^2(\Z^2)}^2 \, = \, 
\| \, D_{1,+} \, D_{2,0} \, U \, \|_{\ell^2(\Z^2)}^2 \, - \, \dfrac{1}{4} \, \| \, \Delta_1 \, D_{2,0} \, U \, \|_{\ell^2(\Z^2)}^2 \, ,
$$
and by then using \eqref{formulelem2-2} for each of the two terms on the right hand side (since $D_{2,0}$ commutes with both $D_{1,+}$ and $\Delta_1$).
\end{proof}

\noindent We can now prove Lemma \ref{lem1}.

\begin{proof}[Proof of Lemma \ref{lem1}]
Let $u^n \in \ell^2(\Z^2;\R)$, and let the sequences $v^n,w^n \in \ell^2(\Z^2;\R)$ be defined by \eqref{defvwn'}. For simplicity, we introduce the notation 
$\alpha := \lambda \, a$ and $\beta := \mu \, b$. Dropping the $n$ superscript for simplicity, we then compute:
\begin{align*}
\| \, v \, \|_{\ell^2(\Z^2)}^2 \, - \, 2 \, \langle u \, ; \, w \rangle_{\ell^2(\Z^2)} \, =& \, \, 
\alpha^2 \, \| \, D_{1,0} \, u \, \|_{\ell^2(\Z^2)}^2 \, + \, \beta^2 \, \| \, D_{2,0} \, u \, \|_{\ell^2(\Z^2)}^2 \\
&+ \, 2 \, \alpha \, \beta \, \langle D_{1,0 } \, u \, ; \, D_{2,0} \, u \rangle_{\ell^2(\Z^2)} \, + \, 
2 \, \alpha \, \beta \, \langle u \, ; \, D_{1,0 } \, D_{2,0} \, u \rangle_{\ell^2(\Z^2)} \\
&+ \, \alpha^2 \, \langle u \, ; \, \Delta_1 \, u \rangle_{\ell^2(\Z^2)} \, + \, \beta^2 \, \langle u \, ; \, \Delta_2 \, u \rangle_{\ell^2(\Z^2)} 
\, - \, \dfrac{\alpha^2 + \beta^2}{4} \, \langle u \, ; \, \Delta_1 \, \Delta_2 \, u \rangle_{\ell^2(\Z^2)} \, .
\end{align*}
Since the operator $D_{1,0}$ is skew-selfadjoint, we observe that the two terms in the second line of the right hand side cancel each other. In the third 
line, we write $\Delta_1=-(D_{1,+})^* \, D_{1,+}$, $\Delta_2=-(D_{2,+})^* \, D_{2,+}$, and $\Delta_1 \, \Delta_2=(D_{1,+} \, D_{2,+})^* \, D_{1,+} \, D_{2,+}$ 
to obtain:
\begin{align*}
\| \, v \, \|_{\ell^2(\Z^2)}^2 \, - \, 2 \, \langle u \, ; \, w \rangle_{\ell^2(\Z^2)} \, =& \, 
\alpha^2 \, \Big( \| \, D_{1,0} \, u \, \|_{\ell^2(\Z^2)}^2 \, - \, \| \, D_{1,+} \, u \, \|_{\ell^2(\Z^2)}^2 \Big) \\
& + \, \beta^2 \, \Big( \| \, D_{2,0} \, u \, \|_{\ell^2(\Z^2)}^2 \, - \, \| \, D_{2,+} \, u \, \|_{\ell^2(\Z^2)}^2 \Big) 
\, - \, \dfrac{\alpha^2 + \beta^2}{4} \, \| \, D_{1,+} \, D_{2,+} \, u \, \|_{\ell^2(\Z^2)}^2 \, .
\end{align*}
It remains to apply relations \eqref{formulelem2-1} and \eqref{formulelem2-2} from Lemma \ref{lem2} to complete the proof of Lemma \ref{lem1}.
\end{proof}

\noindent We now turn to the proof of Proposition \ref{prop1}.

\begin{proof}[Proof of Proposition \ref{prop1}]
Let $u^n \in \ell^2(\Z^2;\R)$, and let the sequence $w^n \in \ell^2(\Z^2;\R)$ be defined by \eqref{defwjkn'}. We use again the short notation 
$\alpha := \lambda \, a$ and $\beta := \mu \, b$, and drop the $n$ superscript for simplicity. We compute the expression:
\begin{align*}
4 \, \| \, w \, \|_{\ell^2(\Z^2)}^2 \, =& \, \, \alpha^4 \, \| \, \Delta_1 \, u \, \|_{\ell^2(\Z^2)}^2 \, + \, \beta^4 \, \| \, \Delta_2 \, u \, \|_{\ell^2(\Z^2)}^2 
\, + \, 2 \, \alpha^2 \, \beta^2 \, {\color{Magenta} \langle \Delta_1 \, u \, ; \, \Delta_2 \, u \rangle_{\ell^2(\Z^2)}} \\
& + \, {\color{ForestGreen} 4 \, \alpha^2 \, \beta^2 \, \| \, D_{1,0} \, D_{2,0} \, u \, \|_{\ell^2(\Z^2)}^2} 
\, + \, \dfrac{(\alpha^2+\beta^2)^2}{16} \, \| \, \Delta_1 \, \Delta_2 \, u \, \|_{\ell^2(\Z^2)}^2 \\
& - \, \dfrac{\alpha^2+\beta^2}{2} \, {\color{blue} \langle \Delta_1 \, \Delta_2 \, u \, ; \, \alpha^2 \, \Delta_1 \, u \, + \, \beta^2 \, \Delta_2 \, u \rangle_{\ell^2(\Z^2)}} \\
& + \, 4 \, \alpha \, \beta \, \langle D_{1,0} \, D_{2,0} \, u \, ; \, \alpha^2 \, \Delta_1 \, u \, + \, \beta^2 \, \Delta_2 \, u \rangle_{\ell^2(\Z^2)} \\
& - \, (\alpha^2+\beta^2) \, \alpha \, \beta \, {\color{red} \langle D_{1,0} \, D_{2,0} \, u \, ; \, \Delta_1 \, \Delta_2 \, u \rangle_{\ell^2(\Z^2)}} \, .
\end{align*}
For the pink term on the right hand side of the first line, we use the very crude estimate:
$$
2 \, \langle \Delta_1 \, u \, ; \, \Delta_2 \, u \rangle_{\ell^2(\Z^2)} \, \le \,  \| \, \Delta_1 \, u \, \|_{\ell^2(\Z^2)}^2 \, + \, \| \, \Delta_2 \, u \, \|_{\ell^2(\Z^2)}^2 \, .
$$
For the green term in the second line of the right hand side, we use the inequality:
\begin{equation}
\label{inegalite-1}
4 \, \alpha^2 \, \beta^2 \, = \, (\alpha^2+\beta^2)^2 \, - \, (\alpha^2-\beta^2)^2 \, \le \, (\alpha^2+\beta^2)^2 \, ,
\end{equation}
and we then use the relation \eqref{formulelem2-3} of Lemma \ref{lem2}. For the blue term (appearing in the third line of the right hand side), we use the 
relations $\Delta_1 =-(D_{1,+})^* \, D_{1,+}$ and $\Delta_2 =-(D_{2,+})^* \, D_{2,+}$. At last, for the red term in the fifth line of the right hand side, we use 
the facts that $D_{1,0}$ is skew-selfadjoint and that $\Delta_2$ is selfadjoint. Collecting all the contributions, we are led to our first preliminary estimate:
\begin{align}
4 \, \| \, w \, \|_{\ell^2(\Z^2)}^2 \, \le & \, \, (\alpha^2+\beta^2) \, 
\Big( \alpha^2 \, \| \, \Delta_1 \, u \, \|_{\ell^2(\Z^2)}^2 \, + \, \beta^2 \, \| \, \Delta_2 \, u \, \|_{\ell^2(\Z^2)}^2 
\, + \, (\alpha^2+\beta^2) \, \| \, D_{1,+} \, D_{2,+} \, u \, \|_{\ell^2(\Z^2)}^2 \Big) \notag \\
& + \, \dfrac{(\alpha^2+\beta^2)^2}{8} \, \| \, \Delta_1 \, \Delta_2 \, u \, \|_{\ell^2(\Z^2)}^2 \, + \, 
4 \, \alpha \, \beta \, \langle D_{1,0} \, D_{2,0} \, u \, ; \, \alpha^2 \, \Delta_1 \, u \, + \, \beta^2 \, \Delta_2 \, u \rangle_{\ell^2(\Z^2)} \notag \\
& + \, {\color{blue} \dfrac{(\alpha^2+\beta^2)}{4} \, (\alpha^2-\beta^2) \, \Big( 
\| \, D_{2,+} \, \Delta_1 \, u \, \|_{\ell^2(\Z^2)}^2 \, - \, \| \, D_{1,+} \, \Delta_2 \, u \, \|_{\ell^2(\Z^2)}^2\Big)} \label{expressionw} \\
& + \, {\color{blue} (\alpha^2+\beta^2) \, \alpha \, \beta \, \langle D_{1,0} \, \Delta_2 \, u \, ; \, D_{2,0} \, \Delta_1 \, u \rangle_{\ell^2(\Z^2)}} \, .\notag
\end{align}
Let us observe at this stage that the first line on the right hand side of \eqref{expressionw} is precisely half of the right hand side of \eqref{estimprop1}. 
We are, to some extent, half way down the road. The analysis now splits in two steps.
\bigskip

$\bullet$ \underline{Step 1.} The bound from above for the norm of $w$ first relies on an estimate of the blue terms in the third and fourth lines on the 
right hand side of \eqref{expressionw}. For later use, we therefore define the quantity:
\begin{equation}
\label{defA}
A \, := \, (\alpha^2-\beta^2) \, \Big( \| \, D_{2,+} \, \Delta_1 \, u \, \|_{\ell^2(\Z^2)}^2 \, - \, \| \, D_{1,+} \, \Delta_2 \, u \, \|_{\ell^2(\Z^2)}^2 \Big) 
\, + \, 4 \, \alpha \, \beta \, \langle D_{1,0} \, \Delta_2 \, u \, ; \, D_{2,0} \, \Delta_1 \, u \rangle_{\ell^2(\Z^2)} \, .
\end{equation}
The estimate \eqref{expressionw} thus reads:
\begin{align*}
4 \, \| \, w \, \|_{\ell^2(\Z^2)}^2 \, \le & \, \, (\alpha^2+\beta^2) \, 
\Big( \alpha^2 \, \| \, \Delta_1 \, u \, \|_{\ell^2(\Z^2)}^2 \, + \, \beta^2 \, \| \, \Delta_2 \, u \, \|_{\ell^2(\Z^2)}^2 
\, + \, (\alpha^2+\beta^2) \, \| \, D_{1,+} \, D_{2,+} \, u \, \|_{\ell^2(\Z^2)}^2 \Big) \\
& + \, \dfrac{(\alpha^2+\beta^2)^2}{8} \, \| \, \Delta_1 \, \Delta_2 \, u \, \|_{\ell^2(\Z^2)}^2 \, + \, 
4 \, \alpha \, \beta \, \langle D_{1,0} \, D_{2,0} \, u \, ; \, \alpha^2 \, \Delta_1 \, u \, + \, \beta^2 \, \Delta_2 \, u \rangle_{\ell^2(\Z^2)} \\
& + \, \dfrac{(\alpha^2+\beta^2)}{4} \, A \, .
\end{align*}
By the Cauchy-Schwarz inequality, the definition \eqref{defA} gives:
$$
A \, \le \, (\alpha^2-\beta^2) \, \Big( \| \, D_{2,+} \, \Delta_1 \, u \, \|_{\ell^2(\Z^2)}^2 \, - \, \| \, D_{1,+} \, \Delta_2 \, u \, \|_{\ell^2(\Z^2)}^2 \Big) 
\, + \, 4 \, |\alpha| \, |\beta| \, \|\, D_{1,0} \, \Delta_2 \, u \, \|_{\ell^2(\Z^2)} \, \| \, D_{2,0} \, \Delta_1 \, u \|_{\ell^2(\Z^2)} \, .
$$
We then use the estimate:
$$
4 \, a_1 \, a_2 \, a_3 \, a_4 \, \le \, \big( a_1^2 \, + \, a_2^2 \big) \, \big( a_3^2 \, + \, a_4^2 \big) \, ,
$$
that is valid for any four real numbers $a_1,a_2,a_3,a_4$, and obtain:
\begin{align*}
A \, \le \, & \, (\alpha^2-\beta^2) \, \Big( \| \, D_{2,+} \, \Delta_1 \, u \, \|_{\ell^2(\Z^2)}^2 \, - \, \| \, D_{1,+} \, \Delta_2 \, u \, \|_{\ell^2(\Z^2)}^2 \Big) \\
&+ \, (\alpha^2+\beta^2) \, \Big( \| \, D_{1,0} \, \Delta_2 \, u \, \|_{\ell^2(\Z^2)}^2 \, + \, \| \, D_{2,0} \, \Delta_1 \, u \, \|_{\ell^2(\Z^2)}^2 \Big) \, .
\end{align*}
We now use the formulas \eqref{formulelem2-1} and \eqref{formulelem2-2} to get the estimate:
$$
A \, \le \, 2 \, \alpha^2 \, \| \, D_{2,+} \, \Delta_1 \, u \, \|_{\ell^2(\Z^2)}^2 \, + \, 2 \, \beta^2 \, \| \, D_{1,+} \, \Delta_2 \, u \, \|_{\ell^2(\Z^2)}^2 
\, - \, \dfrac{(\alpha^2+\beta^2)}{2} \, \| \, \Delta_1 \, \Delta_2 \, u \, \|_{\ell^2(\Z^2)}^2 \, .
$$
We use this estimate of the term $A$ in \eqref{expressionw} to derive the intermediate estimate:
\begin{align}
4 \, \| \, w \, \|_{\ell^2(\Z^2)}^2 \, \le & \, \, (\alpha^2+\beta^2) \, 
\Big( \alpha^2 \, \| \, \Delta_1 \, u \, \|_{\ell^2(\Z^2)}^2 \, + \, \beta^2 \, \| \, \Delta_2 \, u \, \|_{\ell^2(\Z^2)}^2 
\, + \, (\alpha^2+\beta^2) \, \| \, D_{1,+} \, D_{2,+} \, u \, \|_{\ell^2(\Z^2)}^2 \Big) \notag \\
& + \, \dfrac{\alpha^2+\beta^2}{2} \, \Big( 
\alpha^2 \, \| \, D_{2,+} \, \Delta_1 \, u \, \|_{\ell^2(\Z^2)}^2 \, + \, \beta^2 \, \| \, D_{1,+} \, \Delta_2 \, u \, \|_{\ell^2(\Z^2)}^2 \Big) \label{estimw} \\
& + \, 4 \, \alpha \, \beta \, \langle D_{1,0} \, D_{2,0} \, u \, ; \, \alpha^2 \, \Delta_1 \, u \, + \, \beta^2 \, \Delta_2 \, u \rangle_{\ell^2(\Z^2)} \, .\notag
\end{align}
\bigskip

$\bullet$ \underline{Step 2.} We now define the quantities:
\begin{subequations}
\label{defB}
\begin{align}
B_1 \, &:= \, \dfrac{\alpha^2+\beta^2}{2} \, \| \, D_{2,+} \, \Delta_1 \, u \, \|_{\ell^2(\Z^2)}^2 \, + \, 
4 \, \alpha \, \beta \, \langle D_{1,0} \, D_{2,0} \, u \, ; \, \Delta_1 \, u \rangle_{\ell^2(\Z^2)} \, ,\label{defB1} \\
B_2 \, &:= \, \dfrac{\alpha^2+\beta^2}{2} \, \| \, D_{1,+} \, \Delta_2 \, u \, \|_{\ell^2(\Z^2)}^2 \, + \, 
4 \, \alpha \, \beta \, \langle D_{1,0} \, D_{2,0} \, u \, ; \, \Delta_2 \, u \rangle_{\ell^2(\Z^2)} \, .\label{defB2}
\end{align}
\end{subequations}
The estimate \eqref{estimw} thus reads:
\begin{align}
4 \, \| \, w \, \|_{\ell^2(\Z^2)}^2 \, \le& \, \, (\alpha^2+\beta^2) \, 
\Big( \alpha^2 \, \| \, \Delta_1 \, u \, \|_{\ell^2(\Z^2)}^2 \, + \, \beta^2 \, \| \, \Delta_2 \, u \, \|_{\ell^2(\Z^2)}^2 
\, + \, (\alpha^2+\beta^2) \, \| \, D_{1,+} \, D_{2,+} \, u \, \|_{\ell^2(\Z^2)}^2 \Big) \label{estimw-prime} \\
& + \, \alpha^2 \, B_1 \, + \, \beta^2 \, B_2 \, .\notag
\end{align}
The final task is to analyze the terms $B_1$ and $B_2$ defined in \eqref{defB} in order to derive the final estimate \eqref{estimprop1}. 
To do so, we define the following average operators on $\ell^2(\Z^2;\R)$:
\begin{align*}
&(A_{1,+}u)_{j,k} \, := \, \dfrac{u_{j+1,k} \, + \, u_{j,k}}{2} \, ,\quad (A_{1,-}u)_{j,k} \, := \, \dfrac{u_{j,k} \, + \, u_{j-1,k}}{2}  \, ,\\
&(A_{2,+}u)_{j,k} \, := \, \dfrac{u_{j,k+1} \, + \, u_{j,k}}{2} \, ,\quad (A_{2,-}u)_{j,k} \, := \, \dfrac{u_{j,k} \, + \, u_{j,k-1}}{2}  \, .
\end{align*}
It is not hard to see that for the $\ell^2(\Z^2;\R)$ scalar product, there holds $(A_{1,+})^*=A_{1,-}$ and $(A_{2,+})^*=A_{2,-}$. 
Furthermore, we have the relations:
$$
D_{1,0} \, = \, D_{1,+} \, A_{1,-} \, = \, D_{1,-} \, A_{1,+} \, ,\quad D_{2,0} \, = \, D_{2,+} \, A_{2,-} \, = \, D_{2,-} \, A_{2,+} \, ,
$$
and all operators commute. Using the definition \eqref{defB1} and the relation $(A_{2,-})^*=A_{2,+}$, we get:
\begin{align*}
B_1 \, = \, & \, \dfrac{\alpha^2+\beta^2}{2} \, \| \, D_{2,+} \, \Delta_1 \, u \, \|_{\ell^2(\Z^2)}^2 \, + \, 
4 \, \alpha \, \beta \, \langle D_{1,0} \, D_{2,+} \, u \, ; \, A_{2,+} \, \Delta_1 \, u \rangle_{\ell^2(\Z^2)} \\
= \, & \, \dfrac{\alpha^2+\beta^2}{2} \, \| \, D_{2,+} \, \Delta_1 \, u \, \|_{\ell^2(\Z^2)}^2 \, + \, 
4 \, \alpha \, \beta \, \langle A_{1,-} \, D_{1,+} \, D_{2,+} \, u \, ; \, A_{2,+} \, \Delta_1 \, u \rangle_{\ell^2(\Z^2)} \\
\le \, & \, \dfrac{\alpha^2+\beta^2}{2} \, \| \, D_{2,+} \, \Delta_1 \, u \, \|_{\ell^2(\Z^2)}^2 \, + \, 
4 \, |\, \alpha \, | \, |\, \beta \,|\, \| \, A_{1,-} \, D_{1,+} \, D_{2,+} \, u \, \|_{\ell^2(\Z^2)} \, \| \, A_{2,+} \, \Delta_1 \, u \, \|_{\ell^2(\Z^2)}  \, .
\end{align*}
In particular, we get:
\begin{equation}
\label{premiereestimB1}
B_1 \, \le \, \dfrac{\alpha^2+\beta^2}{2} \, \| \, D_{2,+} \, \Delta_1 \, u \, \|_{\ell^2(\Z^2)}^2 \, + \, (\alpha^2+\beta^2) \, 
\Big( \| \, A_{1,-} \, D_{1,+} \, D_{2,+} \, u \, \|_{\ell^2(\Z^2)}^2 \, + \, \| \, A_{2,+} \, \Delta_1 \, u \, \|_{\ell^2(\Z^2)}^2 \Big) \, .
\end{equation}

For any sequence $W \in \ell^2(\Z^2;\R)$, we compute:
$$
\| \, A_{1,-} \, W \, \|_{\ell^2(\Z^2)}^2 \, + \, \dfrac{1}{4} \, \| \, D_{1,-} \, W \, \|_{\ell^2(\Z^2)}^2 \, = \, 
\dfrac{1}{4} \, \sum_{(j,k) \in \Z^2} \, (W_{j,k} \, + \, W_{j-1,k})^2 \, + \, (W_{j,k} \, - \, W_{j-1,k})^2 \, = \, \| \, W \, \|_{\ell^2(\Z^2)}^2 \, ,
$$
and a similar result holds with operators $(A_{1,+},D_{1,+})$ on the left hand side instead of $(A_{1,-},D_{1,-})$ (and similar relations, of course, with 
$(A_{2,+},D_{2,+})$ or $(A_{2,-},D_{2,-})$). Using these relations in \eqref{premiereestimB1}, we end up with our first upper bound for the quantity 
$B_1$ in \eqref{defB1} (the estimate for the quantity $B_2$ in \eqref{defB2} is obtained similarly):
\begin{subequations}
\label{estimB-1}
\begin{align}
B_1 \, & \, \le \, (\alpha^2+\beta^2) \, \Big( \| \, \Delta_1 \, u \, \|_{\ell^2(\Z^2)}^2 \, + \, \| \, D_{1,+} \, D_{2,+} \, u \, \|_{\ell^2(\Z^2)}^2 \Big) \, ,\label{estimB1-1} \\
B_2 \, & \, \le \, (\alpha^2+\beta^2) \, \Big( \| \, \Delta_2 \, u \, \|_{\ell^2(\Z^2)}^2 \, + \, \| \, D_{1,+} \, D_{2,+} \, u \, \|_{\ell^2(\Z^2)}^2 \Big) \, .\label{estimB2-1}
\end{align}
\end{subequations}
Using the estimates \eqref{estimB-1} in \eqref{estimw-prime}, we obtain the estimate \eqref{estimprop1} of Proposition \ref{prop1}.
\end{proof}

\noindent Our final and main result, that actually dates back to \cite{laxwendroff}, is a direct consequence of Lemma \ref{lem1} and Proposition \ref{prop1}. 
The main difference here with \cite{laxwendroff} is that our whole proof is based on the energy method and bypasses Fourier analysis.

\begin{corollary}
\label{cor1}
Let $a,b \in \R$ and let us assume that the CFL parameters $\lambda$ and $\mu$ in \eqref{LW} satisfy:
\begin{equation}
\label{conditionCFL}
(\lambda \, a)^2 +(\mu \, b)^2 \, = \, \dfrac{1 \, - \, \delta}{2} \, ,
\end{equation}
for some real number $\delta \in [0,1]$. Then the two-dimensional Lax-Wendroff scheme \eqref{LW} with an initial condition $u^0 \in \ell^2(\Z^2;\R)$ satisfies:
\begin{equation}
\label{stabilite2D-1}
\forall \, n \in \N \, ,\quad \| \, u^{n+1} \, \|_{\ell^2(\Z^2)}^2 \, - \, \| \, u^n \, \|_{\ell^2(\Z^2)}^2 
\, \le \, - \, \dfrac{\delta}{4} \, \Big( (\lambda \, a)^2 \, \| \, \Delta_1 \, u^n \, \|_{\ell^2(\Z^2)}^2 \, + \, (\mu \, b)^2 \, \| \, \Delta_2 \, u^n \, \|_{\ell^2(\Z^2)}^2 \Big) \, .
\end{equation}
\end{corollary}

\begin{proof}
We first go back to the relation \eqref{energie-cauchy} and apply Lemma \ref{lem1} and Proposition \ref{prop1} for the terms on the right hand side. 
This gives:
\begin{multline*}
\| \, u^{n+1} \, \|_{\ell^2(\Z^2)}^2 \, - \, \| \, u^n \, \|_{\ell^2(\Z^2)}^2 \\
\le \, \Big( - \, 1 \, + \, 2 \, (\lambda \, a)^2 \, + \, 2 \, (\mu \, b)^2 \Big) \times \left( 
\dfrac{(\lambda \, a)^2}{4} \, \| \, \Delta_1 \, u^n \, \|_{\ell^2(\Z^2)}^2 \, + \, \dfrac{(\mu \, b)^2}{4} \, \| \, \Delta_2 \, u^n \, \|_{\ell^2(\Z^2)}^2 \right. \\
\left. + \, \dfrac{(\lambda \, a)^2 + (\mu \, b)^2}{4} \, \| \, D_{1,+} \, D_{2,+}  \, u^n \, \|_{\ell^2(\Z^2)}^2 \right) \, .
\end{multline*}
The result of Corollary \ref{cor1} follows immediately.
\end{proof}

The first consequence of Corollary \ref{cor1} is that, under the CFL condition \eqref{conditionCFL}, that is, when \eqref{CFLlaxwendroff} holds, the 
scheme \eqref{LW} is stable in $\ell^2(\Z^2;\R)$ since the sequence $( \| \, u^n \, \|_{\ell^2(\Z^2)} )_{n \in \N}$ is nonincreasing. We thus recover 
the optimal stability criterion \eqref{CFLlaxwendroff} without resorting to the amplification factor of \eqref{LW}. Moreover, if $\delta$ is positive in 
\eqref{conditionCFL}, that is, if there holds the strict condition:
$$
(\lambda \, a)^2 +(\mu \, b)^2 \, < \, \dfrac{1}{2} \, ,
$$
with furthermore $a \neq 0$ and $b \neq 0$, then Corollary \ref{cor1} shows that the scheme \eqref{LW} is dissipative of order $4$ (in the sense of 
\cite[Definition 5.2.1]{gko}). The crucial point of the above analysis is that Lemma \ref{lem1} and Proposition \ref{prop1} are obtained by using (discrete) 
integration by parts and Cauchy-Schwarz inequalities. The method of proof bypasses Fourier analysis and thus has a chance to extend to more general 
spatial domains. This is precisely such an extension which we explore below, first in the half-space geometry and then in the quarter-space.

\section{The half-space case with the Neumann boundary condition}
\label{section3}

\subsection{The main result}

The method we have used in Lemma \ref{lem1} and Proposition \ref{prop1} will now guide us in our analysis of the Lax-Wendroff scheme \eqref{LW} 
in a half space with the so-called extrapolation (or Neumann) boundary condition. Namely, we now consider the following \emph{outgoing} transport 
equation in a half-space:
\begin{equation}
\label{transportsortant}
\begin{cases}
\partial_t u \, + \, a \, \partial_x u \, + \, b \, \partial_y u \, = \, 0 \, ,& t \ge 0 \, ,\, (x,y) \in \R^+ \times \R \, ,\\
u_{|_{t=0}} \, = \, u_0 \, ,& 
\end{cases}
\end{equation}
where we assume $a<0$, so that no boundary condition is required at the boundary $\{ x=0 \}$ of the space domain.

For any $j \in \Z$, the numerical scheme \eqref{LW} requires the knowledge of the $u_{j-1,k}^n$'s, $k \in \Z$, in order to determine the $u_{j,k}^{n+1}$'s. 
For the half-space problem \eqref{transportsortant}, we use the discrete set of indices $\mathcal{I} := \N \times \Z$ for the \emph{interior} values of the 
numerical solution. We also use the notation $\mathcal{J} := (\{ -1 \} \cup \N) \times \Z$ for the full set of indices corresponding to the cells on which 
the numerical solution is defined (including the so-called \emph{ghost cells} that corresponds to $j=-1$ in our notation). The grid is depicted on Figure 
\ref{fig:maillage1} below.

\begin{figure}[htbp]
\begin{center}
\begin{tikzpicture}[scale=1.5,>=latex]
\draw [ultra thin, dotted, fill=blue!20] (-3,-2) rectangle (3.5,2);
\draw [ultra thin, dotted, fill=red!20] (-3.5,-2) rectangle (-3,2);
\draw [thin, dashed] (-3.5,-2) grid [step=0.5] (3.5,2);
\draw[thick,black,->] (-4,0) -- (4,0) node[below] {$x$};
\draw[thick,black,->] (-3,-2.2)--(-3,2.5) node[right] {$y$};
\draw (-3.73,-0.38) node[right, fill=red!20]{$-\Delta y$};
\draw (-3.9,-0.88) node[right, fill=red!20]{$-2 \, \Delta y$};
\draw (-3.7,1.15) node[right, fill=red!20]{$k \, \Delta y$};
\draw [thin, dashed] (-3.5,1) grid [step=0.5] (-3,1.5);
\draw (-3.6,-0.18) node{$-\Delta x$};
\draw (-3.18,-0.18) node[right]{$\, \, \, 0$};
\draw (-2.6,-0.18) node{$\, \, \, \, \, \Delta x$};
\draw (2.35,-0.18) node{$\, \, \, \, \, j \, \Delta x$};
\node (centre) at (-3.5,0){$\times$};
\node (centre) at (-3,0){$\times$};
\node (centre) at (-2.5,0){$\times$};
\node (centre) at (2.5,0){$\times$};
\node (centre) at (-3,1){$\times$};
\node (centre) at (-3,-0.5){$\times$};
\node (centre) at (-3,-1){$\times$};
\draw (2.75,1.25) node{$u_{j,k}^n$};
\end{tikzpicture}
\caption{The grid for the half-space problem. Interior cells appear in blue and the boundary (ghost) cells appear in red. The value $u_{j,k}^n$ corresponds 
to the approximation in the cell $[j \, \Delta x,(j+1) \, \Delta x) \times [k \, \Delta y,(k+1) \, \Delta y)$.}
\label{fig:maillage1}
\end{center}
\end{figure}
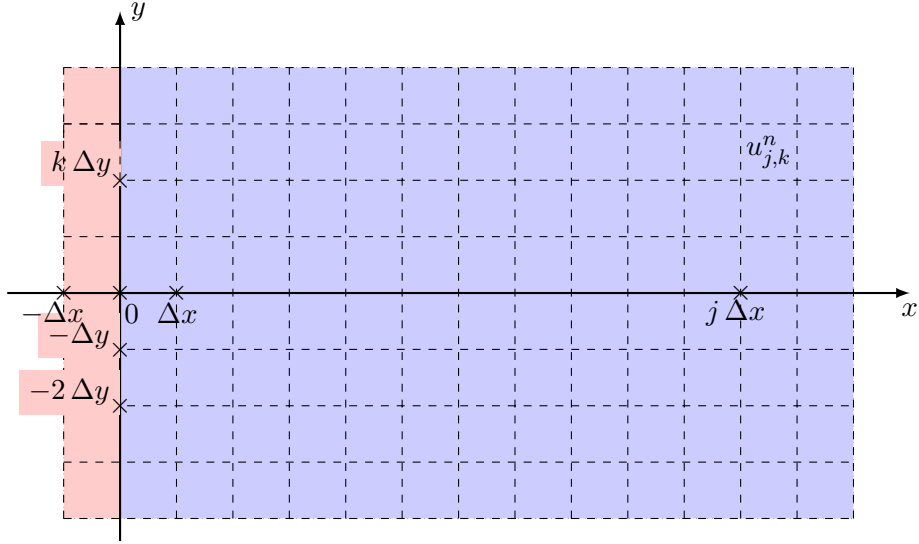

We consider from now on the so-called extrapolation (or Neumann) boundary condition:
\begin{equation}
\label{extrapolation}
\forall \, n \in \N \, ,\quad \forall \, k \in \Z \, ,\quad u_{-1,k}^n \, = \, u_{0,k}^n \, ,
\end{equation}
in conjunction with the numerical scheme \eqref{LW} for $(j,k) \in \mathcal{I}$ (that is, for interior values).

It is useful below to define the Hilbert space $\mathcal{H}$ of real valued, square integrable sequences on $\mathcal{J}$ that satisfy the extrapolation 
boundary condition \eqref{extrapolation}, that is:
\begin{equation}
\label{defH}
\mathcal{H} \, := \, \Big\{ u \in \ell^2(\mathcal{J};\R) \quad |\quad \forall \, k \in \Z \, ,\quad u_{-1,k} \, = \, u_{0,k} \Big\} \, .
\end{equation}
The norm on $\mathcal{H}$ is defined as follows:
\begin{equation}
\label{normH}
\forall \, u \in \mathcal{H} \, ,\quad \| \, u \, \|_{\mathcal{H}}^2 \, := \, \sum_{(j,k) \in \mathcal{I}} \, u_{j,k}^2 \, ,
\end{equation}
that is, we only use interior values of $u$ to compute the $\ell^2$ norm. One easily verifies that the space $\mathcal{H}$ equipped with the above norm 
is a Hilbert space. Given $u^n \in \mathcal{H}$, our numerical approximation of the solution to \eqref{transportsortant} consists in determining $u^{n+1}$ 
by imposing \eqref{LW} in the interior cells indexed by $\mathcal{I}$ (that is, for $j \in \N$) and by requiring $u^{n+1} \in \mathcal{H}$ to determine the 
values in the ghost cells, that is by imposing \eqref{extrapolation} at the following time step:
$$
\forall \, k \in \Z \, ,\quad u_{-1,k}^{n+1} \, = \, u_{0,k}^{n+1} \, .
$$
In this setting, our main result is the following.

\begin{theorem}
\label{thm1}
Let $a<0$, $b \in \R$, and let the parameters $\lambda$, $\mu$ satisfy the stability condition \eqref{CFLlaxwendroff}. Then the numerical scheme 
consisting of \eqref{LW} on $\mathcal{I}$ with the extrapolation numerical boundary condition \eqref{extrapolation} and an initial condition $u^0 \in 
\mathcal{H}$ satisfies the following property: 
for any $n \in \N$, there holds:
$$
\|\, u^{n+1} \, \|_{\mathcal{H}}^2 \, - \, \|\, u^n \, \|_{\mathcal{H}}^2 \, + \, \dfrac{\lambda \, | \, a \, |}{2} \, \sum_{k \in \Z} \, (u_{0,k}^n)^2 \, 
 \, + \, (\mu \, b)^2 \, \dfrac{(\lambda \, a)^2 + (\mu \, b)^2}{16} \, \sum_{k \in \Z} \, (\Delta_2 \, u_{0,k}^n)^2 \, \le \, 0 \, .
$$
In particular, summing with respect to $n \in \N$, there holds:
$$
\sup_{n \in \N}\, \|\, u^n \, \|_{\mathcal{H}}^2 \, + \, \dfrac{\lambda \, | \, a \, |}{2} \, \sum_{n \in \N} \, \sum_{k \in \Z} \, (u_{0,k}^n)^2 
 \, + \, (\mu \, b)^2 \, \dfrac{(\lambda \, a)^2 + (\mu \, b)^2}{16} \, \sum_{n \in \N} \, \sum_{k \in \Z} \, (\Delta_2 \, u_{0,k}^n)^2 \, \le \, 
 2 \, \|\, u^0 \, \|_{\mathcal{H}}^2 \, .
$$
\end{theorem}

\noindent Actually, we could recover some dissipation term (as in the right hand side of \eqref{stabilite2D-1}) if \eqref{CFLlaxwendroff} were satisfied with 
a strict inequality, but we rather focus here on the fact that we get a trace estimate for the numerical solution (and some extra control of the tangential 
Laplacian of the trace provided that the tangential velocity $b$ is nonzero). The trace estimate provided by Theorem \ref{thm1} strongly suggests that 
the Uniform Kreiss-Lopatinskii Condition is satisfied\footnote{The verification of the Uniform Kreiss-Lopatinskii Condition requires considering nonhomogeneous 
boundary conditions, which would give rise to many additional terms.}. The rest of this section is devoted to the proof of Theorem \ref{thm1}. We follow 
the lines of the proofs of Lemma \ref{lem1} and Proposition \ref{prop1}.

\subsection{Proof of Theorem \ref{thm1}}
\label{paragraphpreuve}

We extend the method of Section \ref{section2} to the numerical scheme defined by \eqref{LW} on $\mathcal{I}$ with the numerical boundary condition 
\eqref{extrapolation}. In Section \ref{section2}, the starting point was the energy balance \eqref{energie-cauchy} that used the orthogonality in $\ell^2 
(\Z^2;\R)$ of $v^n$ with respect to both $u^n$ and $w^n$. We first clarify how this extends to the half-space problem we consider here.

\begin{lemma}
\label{lem4}
Let $a<0$ and $b \in \R$. Let $u^n \in \mathcal{H}$, and let the sequences $v^n,w^n$ be defined on the set of interior indices $\mathcal{I}$ by \eqref{defvwn'}. 
Then there holds:
\begin{subequations}
\label{termes-antisym}
\begin{align}
2 \, \langle u^n \, ; \, v^n \rangle_{\ell^2(\mathcal{I})} \, = \, & \, - \, \lambda \, |\, a \,| \, \sum_{k \in \Z} \, (u_{0,k}^n)^2 \, ,\label{lem4-antisym1} \\
2 \, \langle v^n \, ; \, w^n \rangle_{\ell^2(\mathcal{I})} \, = \, & \, - \, \dfrac{\lambda \, |\, a \,| \, (\mu \, b)^2}{2} \, 
\sum_{k \in \Z} \, u_{0,k}^n \, \Delta_2 \, u_{0,k+1}^n \, .\label{lem4-antisym2}
\end{align}
\end{subequations}
\end{lemma}

\noindent The proof of Lemma \ref{lem4} will use the following result which we state independently for the sake of clarity. Lemma \ref{lem5} is an extension 
of the properties $D_{1,0}^*=-D_{1,0}$ and $\Delta_1^*=\Delta_1$ that we used in the whole space $\Z^2$.

\begin{lemma}
\label{lem5}
Let $U,V \in \mathcal{H}$. Then there hold the relations:
\begin{subequations}
\label{relations-lem5}
\begin{align}
\langle U \, ; \, \Delta_1 \, V \rangle_{\ell^2(\mathcal{I})} \, = \, & \, \langle \Delta_1 \, U \, ; \, V \rangle_{\ell^2(\mathcal{I})} \, ,\label{lem5-1} \\
\langle D_{1,0} \, U \, ; \, \Delta_1 \, U \rangle_{\ell^2(\mathcal{I})} \, = \, & \, 0 \, ,\label{lem5-2} \\
\langle U \, ; \, \Delta_1 \, U \rangle_{\ell^2(\mathcal{I})} \, = \, & \, - \, \| \, D_{1,+} \, U \, \|_{\ell^2(\mathcal{I})}^2 \, ,\label{lem5-3} \\
\|\, D_{1,0} \, U \, \|_{\ell^2(\mathcal{I})} \, = \, & \, \| \, D_{1,+} \, U \, \|_{\ell^2(\mathcal{I})}^2 \, - \, \dfrac{1}{4} \, 
\| \, \Delta_1 \, U \, \|_{\ell^2(\mathcal{I})}^2 \, .\label{lem5-4}
\end{align}
\end{subequations}
\end{lemma}

\begin{proof}[Proof of Lemma \ref{lem5}]
Let $U,V \in \mathcal{H}$. We compute:
$$
\langle U \, ; \, \Delta_1 \, V \rangle_{\ell^2(\mathcal{I})} \, - \, \langle \Delta_1 \, U \, ; \, V \rangle_{\ell^2(\mathcal{I})} \, = \, 
\sum_{j \ge 0,k \in \Z} \, U_{j,k} \, V_{j+1,k} \, - \, U_{j-1,k} \, V_{j,k} \, + \, \sum_{j \ge 0,k \in \Z} \, U_{j,k} \, V_{j-1,k} \, - \, U_{j+1,k} \, V_{j,k} \, .
$$
By performing a change of indices, we find that each sum on the right hand side is telescopic and we get:
$$
\langle U \, ; \, \Delta_1 \, V \rangle_{\ell^2(\mathcal{I})} \, - \, \langle \Delta_1 \, U \, ; \, V \rangle_{\ell^2(\mathcal{I})} \, = \, 
- \, \sum_{k \in \Z} \, U_{-1,k} \, V_{0,k} \, + \, \sum_{k \in \Z} \, U_{0,k} \, V_{-1,k} \, .
$$
The latter right hand side vanishes for $U,V \in \mathcal{H}$ since we have $U_{-1,k} =U_{0,k}$ and $V_{-1,k} =V_{0,k}$ for any $k \in \Z$. 
This completes the proof of \eqref{lem5-1}.

We now prove \eqref{lem5-2}. Let $U \in \mathcal{H}$. For any $(j,k) \in \mathcal{I}$, we have:
\begin{align*}
(D_{1,0} \, U_{j,k}) \, (\Delta_1 \, U_{j,k}) \, &= \, \dfrac{1}{2} \, (D_{1,+} \, U_{j,k} \, + \, D_{1,-} \, U_{j,k}) \, (D_{1,+} \, U_{j,k} \, - \, D_{1,-} \, U_{j,k}) \\
&= \, \dfrac{1}{2} \, (D_{1,+} \, U_{j,k})^2 \, - \, \dfrac{1}{2} \, (D_{1,-} \, U_{j,k})^2 \, = \, \dfrac{1}{2} \, D_{1,+} \, \Big\{(D_{1,-} \, U_{j,k})^2 \Big\} \, .
\end{align*}
Summing with respect to $j \in \N$ and $k \in \Z$, we end up with:
$$
\langle D_{1,0} \, U \, ; \, \Delta_1 \, U \rangle_{\ell^2(\mathcal{I})} \, = \, - \, \dfrac{1}{2} \, \sum_{k \in \Z} \, (D_{1,-} \, U_{0,k})^2 
 \, = \, - \, \dfrac{1}{2} \, \sum_{k \in \Z} \, (U_{0,k} \, - \, U_{-1,k})^2 \, = \, 0 \, ,
$$
which completes the proof of \eqref{lem5-2}.

For \eqref{lem5-3}, we consider $U \in \mathcal{H}$. Then for $(j,k) \in \mathcal{J}$, we define $V_{j,k} :=D_{1,+}U_{j,k}$. We compute:
\begin{align*}
\langle U \, ; \, \Delta_1 \, U \rangle_{\ell^2(\mathcal{I})} \, = \, \sum_{(j,k) \in \mathcal{I}} \, U_{j,k} \, (V_{j,k} \, - \, V_{j-1,k}) 
 \, &= \, \sum_{(j,k) \in \mathcal{I}} \, U_{j,k} \, V_{j,k} \, - \, \sum_{(j,k) \in \mathcal{J}} \, U_{j+1,k} \, V_{j,k} \\
 &= \, - \, \sum_{(j,k) \in \mathcal{I}} \, V_{j,k}^2 \, - \, \sum_{k \in \Z} \, U_{0,k} \, V_{-1,k} \, ,
\end{align*}
and the conclusion follows because $V_{-1,k}$ vanishes for any $k \in \Z$ (since $U$ belongs to $\mathcal{H}$).

Eventually, for the relation \eqref{lem5-4}, we start from the algebraic relation (see \eqref{equationalgebrique}):
$$
\dfrac{(U_{j+1,k} \, - \, U_{j-1,k})^2}{4} \, + \, \dfrac{(U_{j+1,k} \, - \, 2 \, U_{j,k} \, + \, U_{j-1,k})^2}{4} 
\, = \, \dfrac{1}{2} \, \big( U_{j+1,k} \, - \, U_{j,k} \big)^2 \, + \, \dfrac{1}{2} \, \big( U_{j,k} \, - \, U_{j-1,k} \big)^2 \, ,
$$
we sum with respect to $k \in \Z$ and $j \in \N$, and obtain (using $U \in \mathcal{H}$):
$$
\|\, D_{1,0} \, U \, \|_{\ell^2(\mathcal{I})} \, + \, \dfrac{1}{4} \, \| \, \Delta_1 \, U \, \|_{\ell^2(\mathcal{I})}^2 \, = \, \| \, D_{1,+} \, U \, \|_{\ell^2(\mathcal{I})}^2 
\, + \, \dfrac{1}{2} \, \sum_{k \in \Z} \, \big( \underbrace{U_{0,k} \, - \, U_{-1,k}}_{= \, 0} \big)^2 \, ,
$$
which completes the proof of \eqref{lem5-4}.
\end{proof}

\noindent Let us now proceed with the proof of Lemma \ref{lem4}.

\begin{proof}[Proof of Lemma \ref{lem4}]
Let $u^n \in \mathcal{H}$, and let the sequences $v^n,w^n$ be defined on the set of indices $\mathcal{I}$ by \eqref{defvwn'}. We omit the superscript 
$n$ below for the sake of clarity. For $(j,k) \in \mathcal{I}$, we have:
$$
2 \, u_{j,k} \, (D_{1,0} \, u_{j,k}) \, = \, u_{j,k} \, (u_{j+1,k} \, - \, u_{j-1,k}) \, = \, D_{1,+} \, \Big\{ u_{j-1,k} \, u_{j,k} \Big\} \, .
$$
Arguing similarly with $D_{2,0}$ rather than $D_{1,0}$, we get the relation:
$$
\forall \, (j,k) \in \mathcal{I} \, ,\quad 2 \, u_{j,k} \, v_{j,k} \, = \, 
- \, \lambda \, a \, D_{1,+} \, \Big\{u_{j-1,k} \, u_{j,k} \Big\} \, - \, \mu \, b \, D_{2,+} \, \Big\{u_{j,k-1} \, u_{j,k} \Big\} \, .
$$
Summing with respect to $(j,k) \in \mathcal{I}$, we end up with:
$$
2 \, \langle u \, ; \, v \rangle_{\ell^2(\mathcal{I})} \, = \, \lambda \, a \, \sum_{k \in \Z} \, u_{-1,k} \, u_{0,k}  \, = \, \lambda \, a \, \sum_{k \in \Z} \, (u_{0,k})^2 \, ,
$$
where we have used $u_{-1,k}=u_{0,k}$. This gives \eqref{lem4-antisym1} since $a$ is negative.

We now turn to the proof of \eqref{lem4-antisym2} and keep omitting the $n$ superscript for the sake of clarity. We also use the notation $\alpha := \lambda \, a$ 
and $\beta := \mu \, b$. The definitions \eqref{defvwn'} allow us to expand the scalar product of $v$ with $w$ and get:
\begin{align}
2 \, \langle v \, ; \, w \rangle_{\ell^2(\mathcal{I})} \, = \, & \, \alpha^3 \, {\color{ForestGreen} \langle D_{1,0} \, u \, ; \, \Delta_1 \, u \rangle_{\ell^2(\mathcal{I})}} 
\, + \, \beta^3 \, {\color{blue} \langle D_{2,0} \, u \, ; \, \Delta_2 \, u \rangle_{\ell^2(\mathcal{I})}} \notag \\
&\, + \, \alpha \, \beta^2 \, \langle D_{1,0} \, u \, ; \, \Delta_2 \, u \rangle_{\ell^2(\mathcal{I})} 
\, + \, \alpha^2 \, \beta \, \langle D_{2,0} \, u \, ; \, \Delta_1 \, u \rangle_{\ell^2(\mathcal{I})} \label{decompositionvw-1} \\
&\, + \, 2 \, \alpha \, \beta^2 \, \langle D_{2,0} \, u \, ; \, D_{1,0} \, D_{2,0} \, u \rangle_{\ell^2(\mathcal{I})} 
\, + \, 2 \, \alpha^2 \, \beta \, {\color{blue} \langle D_{1,0} \, u \, ; \, D_{2,0} \, D_{1,0} \, u \rangle_{\ell^2(\mathcal{I})}} \notag \\
&\, - \, \dfrac{\alpha^2 + \beta^2}{4} \, \langle \Delta_1 \, \Delta_2 \, u \, ; \, \alpha \, D_{1,0} \, u \, + \, \beta \, D_{2,0} \, u \rangle_{\ell^2(\mathcal{I})} \, .\notag
\end{align}

Let us start with the two blue terms on the right hand side of \eqref{decompositionvw-1}. By applying Fubini's theorem, we can compute the sum over 
$\mathcal{I}$ by first summing with respect to $k \in \Z$ and then summing with respect to $j \in \N$. However, the operator $D_{2,0}$, resp. $\Delta_2$, 
is skew-selfadjoint, resp. selfadjoint, for the $\ell^2(\Z_k)$ scalar product (for any fixed $j$). For instance, given any $j \in \N$, there holds:
$$
\sum_{k \in \Z} \, (D_{2,0} \, u_{j,k}) \, (\Delta_2 \, u_{j,k}) \, = \, 0 \, ,
$$
since $j$ is a parameter and the computation is merely one-dimensional. Applying this argument, we obtain:
$$
\langle D_{2,0} \, u \, ; \, \Delta_2 \, u \rangle_{\ell^2(\mathcal{I})} \, = \, 
\langle D_{1,0} \, u \, ; \, D_{2,0} \, D_{1,0} \, u \rangle_{\ell^2(\mathcal{I})} \, = \, 0 \, ,
$$
that is, the two blue terms on the right hand side of \eqref{decompositionvw-1} vanish. Moreover, Lemma \ref{lem5} shows that the green term on the right 
hand side of \eqref{decompositionvw-1} also vanishes. We thus obtain the simplification:
\begin{align}
2 \, \langle v \, ; \, w \rangle_{\ell^2(\mathcal{I})} \, = \, & \,  \alpha \, \beta^2 \, \langle D_{1,0} \, u \, ; \, \Delta_2 \, u \rangle_{\ell^2(\mathcal{I})} 
\, + \, \alpha^2 \, \beta \, {\color{Magenta} \langle D_{2,0} \, u \, ; \, \Delta_1 \, u \rangle_{\ell^2(\mathcal{I})}} \notag \\
& \, + \, 2 \, \alpha \, \beta^2 \, \langle D_{2,0} \, u \, ; \, D_{1,0} \, D_{2,0} \, u \rangle_{\ell^2(\mathcal{I})} \label{decompositionvw-2} \\
& \, - \, \dfrac{\alpha^2 + \beta^2}{4} \, \alpha \, \langle \Delta_1 \, \Delta_2 \, u \, ; \, D_{1,0} \, u \rangle_{\ell^2(\mathcal{I})} 
\, - \, \dfrac{\alpha^2 + \beta^2}{4} \, \beta \, {\color{Magenta} \langle \Delta_1 \, \Delta_2 \, u \, ; \, D_{2,0} \, u \rangle_{\ell^2(\mathcal{I})}} \, .\notag
\end{align}
Let us now look at the pink terms in \eqref{decompositionvw-2}. Since $u$ belongs to $\mathcal{H}$, we also have $D_{2,0} \, u \in \mathcal{H}$ because 
$D_{2,0}$ acts tangentially with respect to the numerical boundary. Applying Lemma \ref{lem5}, we thus have:
$$
\langle D_{2,0} \, u \, ; \, \Delta_1 \, u \rangle_{\ell^2(\mathcal{I})} \, = \, \langle D_{2,0} \, \Delta_1 \, u \, ; \, u \rangle_{\ell^2(\mathcal{I})} \, .
$$
Then summing first with respect to $k \in \Z$, we can use the skew-selfadjointness of $D_{2,0}$ and get:
$$
\langle D_{2,0} \, u \, ; \, \Delta_1 \, u \rangle_{\ell^2(\mathcal{I})} \, = \, - \, \langle \Delta_1 \, u \, ; \, D_{2,0} \, u \rangle_{\ell^2(\mathcal{I})} \, ,
$$
which means that the first pink term on the right hand side of \eqref{decompositionvw-2} vanishes. An entirely similar argument shows that the second 
pink term on the right hand side of \eqref{decompositionvw-2} also vanishes, and we are left with:
\begin{align}
2 \, \langle v \, ; \, w \rangle_{\ell^2(\mathcal{I})} \, = \, & \,  \alpha \, \beta^2 \, \langle D_{1,0} \, u \, ; \, \Delta_2 \, u \rangle_{\ell^2(\mathcal{I})} 
\, + \, 2 \, \alpha \, \beta^2 \, \langle D_{2,0} \, u \, ; \, D_{1,0} \, D_{2,0} \, u \rangle_{\ell^2(\mathcal{I})} \label{decompositionvw-3} \\
& \, - \, \dfrac{\alpha^2 + \beta^2}{4} \, \alpha \, {\color{Orange} \langle \Delta_1 \, \Delta_2 \, u \, ; \, D_{1,0} \, u \rangle_{\ell^2(\mathcal{I})}} \, .\notag
\end{align}
We now look at the orange term on the right hand side of \eqref{decompositionvw-3}. Summing first with respect to $k \in \Z$, and using $\Delta_2 =
-(D_{2,+})^* \, D_{2,+}$ in $\ell^2(\Z_k)$, we obtain:
\begin{align*}
\langle \Delta_1 \, \Delta_2 \, u \, ; \, D_{1,0} \, u \rangle_{\ell^2(\mathcal{I})} \, = \, - \, 
\langle \Delta_1 \, D_{2,+} \, u \, ; \, D_{1,0} \, D_{2,+} \, u \rangle_{\ell^2(\mathcal{I})} \, &= \, - \, \dfrac{1}{2} \, \sum_{j \ge 0,k \in \Z} \, 
D_{1,+} \, \Big\{(D_{1,-} \, D_{2,+} \, u_{j,k})^2 \Big\} \\
&= \, \dfrac{1}{2} \, \sum_{k \in \Z} \, \big( D_{2,+} \, u_{0,k} \, - \, D_{2,+} \, u_{-1,k} \big)^2 \, = \, 0 \, ,
\end{align*}
where we used in the end that $u$ belongs to $\mathcal{H}$ (hence $u_{0,k}=u_{-1,k}$ for any $k \in \Z$, which implies that $D_{2,+} \, u_{0,k} = 
D_{2,+} \, u_{-1,k}$ for any $k \in \Z$). At this stage, we have obtained the (much simplified) expression:
\begin{equation}
\label{decompositionvw-4}
2 \, \langle v \, ; \, w \rangle_{\ell^2(\mathcal{I})} \, = \, \alpha \, \beta^2 \, \langle D_{1,0} \, u \, ; \, \Delta_2 \, u \rangle_{\ell^2(\mathcal{I})} 
\, + \, 2 \, \alpha \, \beta^2 \, \langle D_{2,0} \, u \, ; \, D_{1,0} \, D_{2,0} \, u \rangle_{\ell^2(\mathcal{I})} \, ,
\end{equation}
and it only remains to compute the two scalar products on the right hand side (these will give boundary terms that contribute to \eqref{lem4-antisym2}).

For any sequence $V \in \ell^2(\mathcal{J};\R)$, we have:
\begin{equation}
\label{expressio-ipp}
2 \, \langle V \, ; \, D_{1,0} \, V \rangle_{\ell^2(\mathcal{I})} \, = \, \sum_{j \ge 0,k \in \Z} \, D_{1,+} \, \Big\{V_{j-1,k} \, V_{j,k} \Big\} 
\, = \, - \, \sum_{k \in \Z} \, V_{-1,k} \, V_{0,k} \, .
\end{equation}
Going back to \eqref{decompositionvw-4}, we first sum with respect to $k \in \Z$ to get the expression:
$$
2 \, \langle v \, ; \, w \rangle_{\ell^2(\mathcal{I})} \, = \, - \, \alpha \, \beta^2 \, \langle D_{1,0} \, D_{2,+} \, u \, ; \, D_{2,+} \, u \rangle_{\ell^2(\mathcal{I})} 
\, + \, 2 \, \alpha \, \beta^2 \, \langle D_{2,0} \, u \, ; \, D_{1,0} \, D_{2,0} \, u \rangle_{\ell^2(\mathcal{I})} \, ,
$$
and we then use \eqref{expressio-ipp} to get:
$$
2 \, \langle v \, ; \, w \rangle_{\ell^2(\mathcal{I})} \, = \, \dfrac{\alpha \, \beta^2}{2} \, \sum_{k \in \Z} \, (D_{2,+} \, u_{-1,k}) \, (D_{2,+} \, u_{0,k}) 
\, - \, \alpha \, \beta^2 \, \sum_{k \in \Z} \, (D_{2,0} \, u_{-1,k}) \, (D_{2,0} \, u_{0,k}) \, .
$$
Since $u$ belongs to $\mathcal{H}$, we thus get:
$$
2 \, \langle v \, ; \, w \rangle_{\ell^2(\mathcal{I})} \, = \, \dfrac{\alpha \, \beta^2}{2} \, \sum_{k \in \Z} \, (D_{2,+} \, u_{0,k})^2 
\, - \, \alpha \, \beta^2 \, \sum_{k \in \Z} \, (D_{2,0} \, u_{0,k})^2 \, .
$$
Using the relation $D_{2,0}=(D_{2,+}+D_{2,-})/2$ and expanding, we get:
$$
2 \, \langle v \, ; \, w \rangle_{\ell^2(\mathcal{I})} \, = \, - \, \dfrac{\alpha \, \beta^2}{2} \, \sum_{k \in \Z} \, D_{2,-} \, u_{0,k} \, D_{2,+} \, u_{0,k} \, 
\, = \, \dfrac{\alpha \, \beta^2}{2} \, \sum_{k \in \Z} \, u_{0,k} \, \underbrace{D_{2,+}^2 \, u_{0,k}}_{= \, \Delta_2 \, u_{0,k+1}} \, .
$$
This gives \eqref{lem4-antisym2} (recalling the definitions of $\alpha$ and $\beta$ and $\alpha<0$).
\end{proof}

\noindent A direct consequence of Lemma \ref{lem4} is the following result (which will be useful in our final argument for proving Theorem \ref{thm1}).

\begin{corollary}
\label{cor2}
Let $a<0$ and $b \in \R$ satisfy the stability condition \eqref{CFLlaxwendroff}. Let $u^n \in \mathcal{H}$, and let the sequences $v^n,w^n$ be defined 
on the set of interior indices $\mathcal{I}$ by \eqref{defvwn'}. Then there holds:
$$
2 \, \langle v^n \, ; \, w^n \rangle_{\ell^2(\mathcal{I})} \, - \, 2 \, \langle u^n \, ; \, v^n \rangle_{\ell^2(\mathcal{I})} \, \ge \, 
\dfrac{\lambda \, |\, a \,|}{2} \, \sum_{k \in \Z} \, (u_{0,k}^n)^2 \, - \, (\mu \, b)^2 \, \dfrac{(\lambda \, a)^2 +(\mu \, b)^2}{16} 
\, \sum_{k \in \Z} \, (\Delta_2 u_{0,k}^n)^2 \, .
$$
\end{corollary}

\begin{proof}
We apply the Cauchy-Schwarz and Young's inequalities on the right hand side of \eqref{lem4-antisym2} to get:
$$
2 \, \langle v^n \, ; \, w^n \rangle_{\ell^2(\mathcal{I})} \, \ge \, - \, \dfrac{(\lambda \, a)^2 \, (\mu \, b)^2}{(\lambda \, a)^2 \, + \, (\mu \, b)^2} 
\, \sum_{k \in \Z} \, (u_{0,k}^n)^2 \, - \, (\mu \, b)^2 \, \dfrac{(\lambda \, a)^2 \, + \, (\mu \, b)^2}{16} \, \sum_{k \in \Z} \, (\Delta_2 \, u_{0,k}^n)^2 \, .
$$
Using Young's inequality and \eqref{CFLlaxwendroff}, we also have:
$$
\lambda \, |\, a \, | \, (\mu \, b)^2 \, \le \, \dfrac{1}{\sqrt{2}} \, (\lambda \, a)^2 \, (\mu \, b)^2 \, + \, \dfrac{1}{2 \, \sqrt{2}} \, (\mu \, b)^2 
\, \le \, \dfrac{1}{2 \, \sqrt{2}} \, \big( (\lambda \, a)^2 \, + \, (\mu \, b)^2 \big) \, \le \, \dfrac{(\lambda \, a)^2 \, + \, (\mu \, b)^2}{2} \, ,
$$
which gives (this bound is far from being sharp but it is sufficient for our purpose):
$$
2 \, \langle v^n \, ; \, w^n \rangle_{\ell^2(\mathcal{I})} \, \ge \, - \, \dfrac{\lambda \, |\, a \, |}{2} \, \sum_{k \in \Z} \, (u_{0,k}^n)^2 
\, - \, (\mu \, b)^2 \, \dfrac{(\lambda \, a)^2 \, + \, (\mu \, b)^2}{16} \, \sum_{k \in \Z} \, (\Delta_2 \, u_{0,k}^n)^2 \, .
$$
Combining with \eqref{lem4-antisym1}, the claim of Corollary \ref{cor2} follows.
\end{proof}

\noindent We now state and prove the analogue of Lemma \ref{lem1}.

\begin{lemma}
\label{lem6}
Let $u^n \in \mathcal{H}$, and let the sequences $v^n,w^n \in \ell^2(\mathcal{I};\R)$ be defined by \eqref{defvwn'}. Then there holds:
\begin{align*}
\| \, v^n \, \|_{\ell^2(\mathcal{I})}^2 \, - \, 2 \, \langle u^n \, ; \, w^n \rangle_{\ell^2(\mathcal{I})} \, = \, 
& \, - \, \dfrac{(\lambda \, a)^2}{4} \, \| \, \Delta_1 \, u^n \, \|_{\ell^2(\mathcal{I})}^2 \, - \, \dfrac{(\mu \, b)^2}{4} \, \| \, \Delta_2 \, u^n \, \|_{\ell^2(\mathcal{I})}^2 \\
& \, - \, \dfrac{(\lambda \, a)^2 +(\mu \, b)^2}{4} \, \| \, D_{1,+} \, D_{2,+} \, u^n \, \|_{\ell^2(\mathcal{I})}^2 \, .
\end{align*}
\end{lemma}

\begin{proof}
We drop the superscript $n$ for simplicity, and use again the notation $\alpha := \lambda \, a$, $\beta := \mu \, b$. We thus consider $u \in \mathcal{H}$. 
Using the definition \eqref{defvwn'}, we compute:
\begin{align*}
\| \, v \, \|_{\ell^2(\mathcal{I})}^2 \, - \, 2 \, \langle u \, ; \, w \rangle_{\ell^2(\mathcal{I})} \, =& \, \, 
\alpha^2 \, \| \, D_{1,0} \, u \, \|_{\ell^2(\mathcal{I})}^2 \, + \, \beta^2 \, \| \, D_{2,0} \, u \, \|_{\ell^2(\mathcal{I})}^2 \\
&+ \, 2 \, \alpha \, \beta \, \langle D_{1,0 } \, u \, ; \, D_{2,0} \, u \rangle_{\ell^2(\mathcal{I})} \, + \, 
2 \, \alpha \, \beta \, \langle u \, ; \, D_{1,0 } \, D_{2,0} \, u \rangle_{\ell^2(\mathcal{I})} \\
&+ \, \alpha^2 \, \langle u \, ; \, \Delta_1 \, u \rangle_{\ell^2(\mathcal{I})} \, + \, \beta^2 \, \langle u \, ; \, \Delta_2 \, u \rangle_{\ell^2(\mathcal{I})} 
\, - \, \dfrac{\alpha^2 + \beta^2}{4} \, \langle u \, ; \, \Delta_1 \, \Delta_2 \, u \rangle_{\ell^2(\mathcal{I})} \, .
\end{align*}

By first summing with respect to $k \in \Z$, we can use the skew-selfadjointness of $D_{2,0}$ on $\ell^2(\Z_k;\R)$ and find that the two terms on the 
second line of the right hand side cancel each other. We thus get:
\begin{align*}
\| \, v \, \|_{\ell^2(\mathcal{I})}^2 \, - \, 2 \, \langle u \, ; \, w \rangle_{\ell^2(\mathcal{I})} \, =& \, \, 
\alpha^2 \, \Big( \| \, D_{1,0} \, u \, \|_{\ell^2(\mathcal{I})}^2 \, + \, \langle u \, ; \, \Delta_1 \, u \rangle_{\ell^2(\mathcal{I})} \Big) \\
& \, + \, \beta^2 \, \Big( \| \, D_{2,0} \, u \, \|_{\ell^2(\mathcal{I})}^2 \, + \, \langle u \, ; \, \Delta_2 \, u \rangle_{\ell^2(\mathcal{I})} \Big) 
\, - \, \dfrac{\alpha^2 + \beta^2}{4} \, \langle u \, ; \, \Delta_1 \, \Delta_2 \, u \rangle_{\ell^2(\mathcal{I})} \, .
\end{align*}
Using the relations \eqref{lem5-3} and \eqref{lem5-4}, we get:
$$
\| \, D_{1,0} \, u \, \|_{\ell^2(\mathcal{I})}^2 \, + \, \langle u \, ; \, \Delta_1 \, u \rangle_{\ell^2(\mathcal{I})} \, = \, - \, \dfrac{1}{4} \, 
\| \, \Delta_1 \, u \, \|_{\ell^2(\mathcal{I})}^2 \, .
$$
Summing first with respect to $k$ and using the one-dimensional analogue of \eqref{formulelem2-2} (there is no boundary term since we sum with 
respect to $k \in \Z$), we also get:
$$
\| \, D_{2,0} \, u \, \|_{\ell^2(\mathcal{I})}^2 \, + \, \langle u \, ; \, \Delta_2 \, u \rangle_{\ell^2(\mathcal{I})} \, = \, - \, \dfrac{1}{4} \, 
\| \, \Delta_2 \, u \, \|_{\ell^2(\mathcal{I})}^2 \, .
$$
We are thus led to the expression:
$$
\| \, v \, \|_{\ell^2(\mathcal{I})}^2 \, - \, 2 \, \langle u \, ; \, w \rangle_{\ell^2(\mathcal{I})} \, = \, 
- \, \dfrac{\alpha^2}{4} \, \| \, \Delta_1 \, u \, \|_{\ell^2(\mathcal{I})}^2 \, - \, \dfrac{\beta^2}{4} \, \| \, \Delta_2 \, u \, \|_{\ell^2(\mathcal{I})}^2 
\, - \, \dfrac{\alpha^2 + \beta^2}{4} \, {\color{blue} \langle u \, ; \, \Delta_1 \, \Delta_2 \, u \rangle_{\ell^2(\mathcal{I})}} \, .
$$

For the last remaining term on the right hand side (in blue), we first sum with respect to $k$ and use the relation $\Delta_2=-(D_{2,+})^* \, D_{2,+}$ 
for the $\ell^2(\Z;\R)$ scalar product. We get:
$$
\| \, v \, \|_{\ell^2(\mathcal{I})}^2 \, - \, 2 \, \langle u \, ; \, w \rangle_{\ell^2(\mathcal{I})} \, = \, 
- \, \dfrac{\alpha^2}{4} \, \| \, \Delta_1 \, u \, \|_{\ell^2(\mathcal{I})}^2 \, - \, \dfrac{\beta^2}{4} \, \| \, \Delta_2 \, u \, \|_{\ell^2(\mathcal{I})}^2 
\, + \, \dfrac{\alpha^2 + \beta^2}{4} \, \langle D_{2,+} \, u \, ; \, \Delta_1 \, D_{2,+} \, u \rangle_{\ell^2(\mathcal{I})} \, ,
$$
and it then remains to use the relation \eqref{lem5-3} of Lemma \ref{lem5} (which is valid since $D_{2,+} \, u$ also belongs to $\mathcal{H}$). We end 
up with the expected relation:
$$
\| \, v \, \|_{\ell^2(\mathcal{I})}^2 \, - \, 2 \, \langle u \, ; \, w \rangle_{\ell^2(\mathcal{I})} \, = \, 
- \, \dfrac{\alpha^2}{4} \, \| \, \Delta_1 \, u \, \|_{\ell^2(\mathcal{I})}^2 \, - \, \dfrac{\beta^2}{4} \, \| \, \Delta_2 \, u \, \|_{\ell^2(\mathcal{I})}^2 
\, - \, \dfrac{\alpha^2 + \beta^2}{4} \, \|\, D_{1,+} \, D_{2,+} \, u \, \|_{\ell^2(\mathcal{I})}^2 \, .
$$
This completes the proof of Lemma \ref{lem6}.
\end{proof}

As in Section \ref{section2}, the main point in the proof of Theorem \ref{thm1} is an estimate of the norm of the term $w^n$ that is defined in 
\eqref{defwjkn'}. The analogue of Proposition \ref{prop1} in the half-space case with the extrapolation numerical boundary condition \eqref{extrapolation} 
is the following result.

\begin{proposition}
\label{prop2}
Let $u^n \in \mathcal{H}$, and let the sequence $w^n \in \ell^2(\mathcal{I};\R)$ be defined by \eqref{defwjkn'}. Then there holds:
\begin{multline}
\label{estimprop2}
4 \, \| \, w^n \, \|_{\ell^2(\mathcal{I})}^2 \, + \, (\mu \, b)^2 \, \dfrac{(\lambda \, a)^2 +(\mu \, b)^2}{2} \, \sum_{k \in \Z} \, (\Delta_2 \, u_{0,k}^n)^2 \, \le \, 
2 \, \big( (\lambda \, a)^2 +(\mu \, b)^2 \big) \, \Big\{ (\lambda \, a)^2 \, \| \, \Delta_1 \, u^n \, \|_{\ell^2(\mathcal{I})}^2 \\
+ \, (\mu \, b)^2 \, \| \, \Delta_2 \, u^n \, \|_{\ell^2(\mathcal{I})}^2 
\, + \, \big( (\lambda \, a)^2 +(\mu \, b)^2 \big) \, \| \, D_{1,+} \, D_{2,+} \, u^n \, \|_{\ell^2(\mathcal{I})}^2 \Big\} \, .
\end{multline}
\end{proposition}

\begin{proof}
The proof follows the method that we have used when proving Proposition \ref{prop1}, except that we need to take boundary terms into account 
when we use integration by parts with respect to the first variable $j$. Once again, we use the notation $\alpha \, := \, \lambda \, a$ and $\beta \, 
:= \, \mu \, b$ and drop the superscript $n$ for simplicity. We first compute:
\begin{align*}
4 \, \| \, w \, \|_{\ell^2(\mathcal{I})}^2 \, =& \, \, \alpha^4 \, \| \, \Delta_1 \, u \, \|_{\ell^2(\mathcal{I})}^2 \, + \, \beta^4 \, \| \, \Delta_2 \, u \, \|_{\ell^2(\mathcal{I})}^2 
\, + \, 2 \, \alpha^2 \, \beta^2 \, {\color{Magenta} \langle \Delta_1 \, u \, ; \, \Delta_2 \, u \rangle_{\ell^2(\mathcal{I})}} \\
& + \, {\color{ForestGreen} 4 \, \alpha^2 \, \beta^2 \, \| \, D_{1,0} \, D_{2,0} \, u \, \|_{\ell^2(\mathcal{I})}^2} 
\, + \, \dfrac{(\alpha^2+\beta^2)^2}{16} \, \| \, \Delta_1 \, \Delta_2 \, u \, \|_{\ell^2(\mathcal{I})}^2 \\
& - \, \dfrac{\alpha^2+\beta^2}{2} \, \langle \Delta_1 \, \Delta_2 \, u \, ; \, \alpha^2 \, \Delta_1 \, u \, + \, \beta^2 \, \Delta_2 \, u \rangle_{\ell^2(\mathcal{I})} \\
& + \, 4 \, \alpha \, \beta \, \langle D_{1,0} \, D_{2,0} \, u \, ; \, \alpha^2 \, \Delta_1 \, u \, + \, \beta^2 \, \Delta_2 \, u \rangle_{\ell^2(\mathcal{I})} \\
& - \, (\alpha^2+\beta^2) \, \alpha \, \beta \, \langle D_{1,0} \, D_{2,0} \, u \, ; \, \Delta_1 \, \Delta_2 \, u \rangle_{\ell^2(\mathcal{I})} \, ,
\end{align*}
and we use the same crude estimate as in the proof of Proposition \ref{prop1} for the pink term on the right hand side. For the green term on the right 
hand side, we first use the inequality \eqref{inegalite-1} and then the relation \eqref{lem5-4} of Lemma \ref{lem5} and combine it with \eqref{formulelem2-2} 
of Lemma \ref{lem2} to obtain the analogue of \eqref{formulelem2-3}, that is:
\begin{multline}
\label{relationprop2-1}
\| \, D_{1,0} \, D_{2,0} \, u \, \|_{\ell^2(\mathcal{I})}^2 \, = \, \| \, D_{1,+} \, D_{2,+} \, u \, \|_{\ell^2(\mathcal{I})}^2 \, 
- \, \dfrac{1}{4} \, \| \, D_{1,+} \, \Delta_2 \, u \, \|_{\ell^2(\mathcal{I})}^2 \, - \, \dfrac{1}{4} \, \| \, D_{2,+} \, \Delta_1 \, u \, \|_{\ell^2(\mathcal{I})}^2 \\
+ \, \dfrac{1}{16} \, \| \, \Delta_1 \, \Delta_2 \, u \, \|_{\ell^2(\mathcal{I})}^2 \, ,
\end{multline}
since $u$ belongs to $\mathcal{H}$. We thus get our first estimate:
\begin{align}
4 \, \| \, w \, \|_{\ell^2(\mathcal{I})}^2 \, \le & \, \, (\alpha^2+\beta^2) \, 
\Big( \alpha^2 \, \| \, \Delta_1 \, u \, \|_{\ell^2(\mathcal{I})}^2 \, + \, \beta^2 \, \| \, \Delta_2 \, u \, \|_{\ell^2(\mathcal{I})}^2 
\, + \, (\alpha^2+\beta^2) \, \| \, D_{1,+} \, D_{2,+} \, u \, \|_{\ell^2(\mathcal{I})}^2 \Big) \nonumber \\
& + \, \dfrac{(\alpha^2+\beta^2)^2}{8} \, \| \, \Delta_1 \, \Delta_2 \, u \, \|_{\ell^2(\mathcal{I})}^2 
\, - \, \dfrac{(\alpha^2+\beta^2)^2}{4} \, \Big( \| \, D_{1,+} \, \Delta_2 \, u \, \|_{\ell^2(\mathcal{I})}^2 \, + \, \| \, D_{2,+} \, \Delta_1 \, u \, \|_{\ell^2(\mathcal{I})}^2 \Big) \nonumber \\
& + \, 4 \, \alpha \, \beta \, \langle D_{1,0} \, D_{2,0} \, u \, ; \, \alpha^2 \, \Delta_1 \, u \, + \, \beta^2 \, \Delta_2 \, u \rangle_{\ell^2(\mathcal{I})} 
\, - \,(\alpha^2+\beta^2) \, \alpha \, \beta \, {\color{red} \langle D_{1,0} \, D_{2,0} \, u \, ; \, \Delta_1 \, \Delta_2 \, u \rangle_{\ell^2(\mathcal{I})}} \label{ineg-symw2} \\
& - \, \dfrac{\alpha^2+\beta^2}{2} \, {\color{blue} \langle \Delta_1 \, \Delta_2 \, u \, ; \, \alpha^2 \, \Delta_1 \, u \, + \, \beta^2 \, \Delta_2 \, u \rangle_{\ell^2(\mathcal{I})}} \, .\nonumber
\end{align}
For the blue factor, we use either the relation \eqref{lem5-3} (for the sequence $\Delta_2 \, u \in \mathcal{H}$) or we first sum with respect to $k \in \Z$ and 
use the relation $\Delta_2 =-(D_{2,+})^* \, D_{2,+}$ in $\ell^2(\Z;\R)$. For the red term, we also sum first with respect to $k$ and use the skew-selfadjointness 
of $D_{2,0}$ together with the selfadjointness of $\Delta_2$. We end up with the analogue of \eqref{expressionw}, that is:
\begin{align}
4 \, \| \, w \, \|_{\ell^2(\mathcal{I})}^2 \, \le & \, \, (\alpha^2+\beta^2) \, 
\Big( \alpha^2 \, \| \, \Delta_1 \, u \, \|_{\ell^2(\mathcal{I})}^2 \, + \, \beta^2 \, \| \, \Delta_2 \, u \, \|_{\ell^2(\mathcal{I})}^2 
\, + \, (\alpha^2+\beta^2) \, \| \, D_{1,+} \, D_{2,+} \, u \, \|_{\ell^2(\mathcal{I})}^2 \Big) \notag \\
& + \, \dfrac{(\alpha^2+\beta^2)^2}{8} \, \| \, \Delta_1 \, \Delta_2 \, u \, \|_{\ell^2(\mathcal{I})}^2 
\,  + \, 4 \, \alpha \, \beta \, \langle D_{1,0} \, D_{2,0} \, u \, ; \, \alpha^2 \, \Delta_1 \, u \, + \, \beta^2 \, \Delta_2 \, u \rangle_{\ell^2(\mathcal{I})} \notag \\
& + \, {\color{blue} \dfrac{(\alpha^2+\beta^2)}{4} \, (\alpha^2-\beta^2) \, \Big( 
\| \, D_{2,+} \, \Delta_1 \, u \, \|_{\ell^2(\mathcal{I})}^2 \, - \, \| \, D_{1,+} \, \Delta_2 \, u \, \|_{\ell^2(\mathcal{I})}^2 \Big)} \label{inegprop2-1} \\
& + \, {\color{blue} (\alpha^2+\beta^2) \, \alpha \, \beta \, \langle D_{1,0} \, \Delta_2 \, u \, ; \, D_{2,0} \, \Delta_1 \, u \rangle_{\ell^2(\mathcal{I})}} \, .\notag
\end{align}

The estimate of the terms in blue on the right hand side of \eqref{inegprop2-1} is carried out with the same procedure as in Step 1 of the proof of 
Proposition \ref{prop1}. The analysis there used the relations \eqref{formulelem2-1}-\eqref{formulelem2-2} of Lemma \ref{lem2}, and we substitue 
here \eqref{lem5-4} for \eqref{formulelem2-1}. All the arguments of Step 1 in the proof of Proposition \ref{prop1} can be reproduced almost word 
for word, with the only modification that consists in computing scalar products and norms on $\mathcal{I}$ rather than on $\Z^2$. We thus feel free 
to use the corresponding estimate in \eqref{inegprop2-1} to obtain (compare with \eqref{estimw}):
\begin{align}
4 \, \| \, w \, \|_{\ell^2(\mathcal{I})}^2 \, \le & \, \, (\alpha^2+\beta^2) \, 
\Big( \alpha^2 \, \| \, \Delta_1 \, u \, \|_{\ell^2(\mathcal{I})}^2 \, + \, \beta^2 \, \| \, \Delta_2 \, u \, \|_{\ell^2(\mathcal{I})}^2 
\, + \, (\alpha^2+\beta^2) \, \| \, D_{1,+} \, D_{2,+} \, u \, \|_{\ell^2(\mathcal{I})}^2 \Big) \notag \\
& + \, \dfrac{\alpha^2+\beta^2}{2} \, \Big( 
\alpha^2 \, \| \, D_{2,+} \, \Delta_1 \, u \, \|_{\ell^2(\mathcal{I})}^2 \, + \, \beta^2 \, \| \, D_{1,+} \, \Delta_2 \, u \, \|_{\ell^2(\mathcal{I})}^2 \Big) \label{inegprop2-2} \\
& + \, 4 \, \alpha \, \beta \, \langle D_{1,0} \, D_{2,0} \, u \, ; \, \alpha^2 \, \Delta_1 \, u \, + \, \beta^2 \, \Delta_2 \, u \rangle_{\ell^2(\mathcal{I})} \, .\notag
\end{align}
\bigskip

It now remains to proceed with the estimate of the second and third lines on the right hand side of \eqref{inegprop2-2}. Following the proof of Proposition 
\ref{prop1}, we define the quantities:
\begin{subequations}
\label{defB-halfpsace}
\begin{align}
B_1 \, &:= \, \dfrac{\alpha^2+\beta^2}{2} \, \| \, D_{2,+} \, \Delta_1 \, u \, \|_{\ell^2(\mathcal{I})}^2 \, + \, 
4 \, \alpha \, \beta \, \langle D_{1,0} \, D_{2,0} \, u \, ; \, \Delta_1 \, u \rangle_{\ell^2(\mathcal{I})} \, ,\label{defB1-halfspace} \\
B_2 \, &:= \, \dfrac{\alpha^2+\beta^2}{2} \, \| \, D_{1,+} \, \Delta_2 \, u \, \|_{\ell^2(\mathcal{I})}^2 \, + \, 
4 \, \alpha \, \beta \, \langle D_{1,0} \, D_{2,0} \, u \, ; \, \Delta_2 \, u \rangle_{\ell^2(\mathcal{I})} \, .\label{defB2-halfspace}
\end{align}
\end{subequations}
The estimate \eqref{inegprop2-2} thus reads:
\begin{align}
4 \, \| \, w \, \|_{\ell^2(\mathcal{I})}^2 \, \le& \, \, (\alpha^2+\beta^2) \, 
\Big( \alpha^2 \, \| \, \Delta_1 \, u \, \|_{\ell^2(\mathcal{I})}^2 \, + \, \beta^2 \, \| \, \Delta_2 \, u \, \|_{\ell^2(\mathcal{I})}^2 
\, + \, (\alpha^2+\beta^2) \, \| \, D_{1,+} \, D_{2,+} \, u \, \|_{\ell^2(\mathcal{I})}^2 \Big) \label{inegprop2-3} \\
& + \, \alpha^2 \, B_1 \, + \, \beta^2 \, B_2 \, .\notag
\end{align}
At this stage, the analysis follows the method that we used in the proof of Proposition \ref{prop1} (see Step 2 in the proof) but the numerical boundary 
implies slightly different arguments in the estimates of $B_1$ and $B_2$. We thus state and prove two separate results in order to make this final point 
of our analysis clear.

\begin{lemma}
\label{lem7}
Let $u \in \mathcal{H}$, $\alpha,\beta \in \R$ and let $B_1$ be defined by \eqref{defB1-halfspace}. Then there holds:
$$
B_1 \, \le \, (\alpha^2+\beta^2) \, \Big( \| \, \Delta_1 \, u \, \|_{\ell^2(\mathcal{I})}^2 \, + \, \| \, D_{1,+} \, D_{2,+} \, u \, \|_{\ell^2(\mathcal{I})}^2 \Big) \, .
$$
\end{lemma}

\begin{proof}[Proof of Lemma \ref{lem7}]
We introduce again the average operators $A_{1,\pm}$ and $A_{2,\pm}$ defined (whenever the formulas make sense) as in Section \ref{section2} by:
\begin{align*}
&(A_{1,+} \, V)_{j,k} \, := \, \dfrac{V_{j+1,k} \, + \, V_{j,k}}{2} \, ,\quad (A_{1,-} \, V)_{j,k} \, := \, \dfrac{V_{j,k} \, + \, V_{j-1,k}}{2}  \, ,\\
&(A_{2,+} \, V)_{j,k} \, := \, \dfrac{V_{j,k+1} \, + \, V_{j,k}}{2} \, ,\quad (A_{2,-} \, V)_{j,k} \, := \, \dfrac{V_{j,k} \, + \, V_{j,k-1}}{2}  \, .
\end{align*}
Writing then $D_{2,0}=A_{2,-}\, D_{2,+}$ and summing first with respect to $k$ in order to use $A_{2,-}^*=A_{2,+}$, we first get:
\begin{align*}
B_1 \, &= \, \dfrac{\alpha^2+\beta^2}{2} \, \| \, D_{2,+} \, \Delta_1 \, u \, \|_{\ell^2(\mathcal{I})}^2 \, + \, 
4 \, \alpha \, \beta \, \langle D_{1,0} \, D_{2,+} \, u \, ; \, A_{2,+} \, \Delta_1 \, u \rangle_{\ell^2(\mathcal{I})} \\
&= \, \dfrac{\alpha^2+\beta^2}{2} \, \| \, D_{2,+} \, \Delta_1 \, u \, \|_{\ell^2(\mathcal{I})}^2 \, + \, 
4 \, \alpha \, \beta \, \langle A_{1,-} \, D_{1,+} \, D_{2,+} \, u \, ; \, A_{2,+} \, \Delta_1 \, u \rangle_{\ell^2(\mathcal{I})} \\
&\le \, \dfrac{\alpha^2+\beta^2}{2} \, \| \, D_{2,+} \, \Delta_1 \, u \, \|_{\ell^2(\mathcal{I})}^2 \, + \, 
(\alpha^2+\beta^2) \, \Big( \|\, A_{1,-} \, D_{1,+} \, D_{2,+} \, u \, \|_{\ell^2(\mathcal{I})}^2 \, + \, 
\| \, A_{2,+} \, \Delta_1 \, u \, \|_{\ell^2(\mathcal{I})}^2 \Big) \, .
\end{align*}

We can then follow the arguments in Section \ref{section2} to obtain:
\begin{align*}
\|\, A_{2,+} \, \Delta_1 \, u \, \|_{\ell^2(\mathcal{I})}^2 \, =& \, \|\, \Delta_1 \, u \, \|_{\ell^2(\mathcal{I})}^2 \, - \, \dfrac{1}{4} \, 
\|\, D_{2,+} \, \Delta_1 \, u \, \|_{\ell^2(\mathcal{I})}^2 \, ,\\
\|\, A_{1,-} \, D_{1,+} \, D_{2,+} \, u \, \|_{\ell^2(\mathcal{I})}^2 \, =& \, \|\, D_{1,+} \, D_{2,+} \, u \, \|_{\ell^2(\mathcal{I})}^2 \, - \, \dfrac{1}{4} \, 
\|\, D_{2,+} \, \Delta_1 \, u \, \|_{\ell^2(\mathcal{I})}^2 \, ,
\end{align*}
where, in the first equality, we have summed with respect to $k \in \Z$ and then with respect to $j \in \N$, while in the second equality, we have 
summed with respect to $j \in \N$ and then with respect to $k \in \Z$ (and we have also used the fact that $D_{1,+} \, D_{2,+} \, u_{-1,k}=0$ for any 
$k \in \Z$ since $u$ belongs to $\mathcal{H}$). The claim of Lemma \ref{lem7} follows.
\end{proof}

\begin{lemma}
\label{lem8}
Let $u \in \mathcal{H}$, $\alpha,\beta \in \R$ and let $B_2$ be defined by \eqref{defB2-halfspace}. Then there holds:
$$
B_2 \, \le \, (\alpha^2+\beta^2) \, \Big( \| \, \Delta_2 \, u \, \|_{\ell^2(\mathcal{I})}^2 \, + \, \| \, D_{1,+} \, D_{2,+} \, u \, \|_{\ell^2(\mathcal{I})}^2 \Big) 
\, - \, \dfrac{\alpha^2+\beta^2}{2} \, \sum_{k \in \Z} \, (\Delta_2 \, u_{0,k})^2 \, .
$$
\end{lemma}

\begin{proof}[Proof of Lemma \ref{lem8}]
We start from the definition \eqref{defB2-halfspace}. Using the relations $D_{1,0}=A_{1,-} \, D_{1,+}$ and $D_{2,0}=A_{2,-} \, D_{2,+}$, we 
have:
$$
B_2 \, = \, \dfrac{\alpha^2+\beta^2}{2} \, \| \, D_{1,+} \, \Delta_2 \, u \, \|_{\ell^2(\mathcal{I})}^2 \, + \, 
4 \, \alpha \, \beta \, \langle A_{1,-} \, A_{2,-} \, D_{1,+} \, D_{2,+} \, u \, ; \, \Delta_2 \, u \rangle_{\ell^2(\mathcal{I})} \, ,
$$
and the scalar product is of the form $\langle A_{1,-} \, W \, ; \, \widetilde{W} \rangle_{\ell^2(\mathcal{I})}$ for a sequence $W$ that satisfies 
$W_{-1,k}=0$ for any $k \in \Z$. We can therefore shift indices and get:
\begin{align*}
B_2 \, &= \, \dfrac{\alpha^2+\beta^2}{2} \, \| \, D_{1,+} \, \Delta_2 \, u \, \|_{\ell^2(\mathcal{I})}^2 \, + \, 
4 \, \alpha \, \beta \, \langle A_{2,-} \, D_{1,+} \, D_{2,+} \, u \, ; \, A_{1,+} \, \Delta_2 \, u \rangle_{\ell^2(\mathcal{I})} \\
&\le \, \dfrac{\alpha^2+\beta^2}{2} \, \| \, D_{1,+} \, \Delta_2 \, u \, \|_{\ell^2(\mathcal{I})}^2 \, + \, 
(\alpha^2+\beta^2) \, \Big( \|\, A_{2,-} \, D_{1,+} \, D_{2,+} \, u \, \|_{\ell^2(\mathcal{I})}^2 \, + \, 
\| \, A_{1,+} \, \Delta_2 \, u \, \|_{\ell^2(\mathcal{I})}^2 \Big) \, .
\end{align*}

The novelty pops up here. Summing first with respect to $k$, we get once again the relation:
\begin{equation}\label{equation_B2}
\|\, A_{2,-} \, D_{1,+} \, D_{2,+} \, u \, \|_{\ell^2(\mathcal{I})}^2 \, + \, \dfrac{1}{4} \, \|\, D_{1,+} \, \Delta_2 \, u \, \|_{\ell^2(\mathcal{I})}^2 
\, = \, \|\, D_{1,+} \, D_{2,+} \, u \, \|_{\ell^2(\mathcal{I})}^2 \, .
\end{equation}
Now, if $U$ belongs to $\ell^2(\mathcal{I};\R)$, we compute:
\begin{align*}
\| \, A_{1,+} \, U \, \|_{\ell^2(\mathcal{I})}^2 \, + \, \dfrac{1}{4} \, \|\, D_{1,+} \, U \, \|_{\ell^2(\mathcal{I})}^2 
\, = \, & \, \dfrac{1}{4} \, \sum_{j \ge 0,k\in \Z} \, (U_{j,k} \, + \, U_{j+1,k})^2 \, + \, (U_{j+1,k} \, - \, U_{j,k})^2 \\
= \, & \, \dfrac{1}{2} \, \sum_{j \ge 0,k\in \Z} \, U_{j,k}^2 \, + \, \dfrac{1}{2} \, \sum_{j \ge 1,k\in \Z} \, U_{j,k}^2 \\
= \, & \, \| \, U \, \|_{\ell^2(\mathcal{I})}^2 \, - \, \dfrac{1}{2} \, \sum_{k\in \Z} \, U_{0,k}^2 \, .
\end{align*}
Using the above two relations in our estimate of $B_2$, the claim of Lemma \ref{lem8} follows.
\end{proof}

Using the bounds for $B_1$ and $B_2$ given in Lemma \ref{lem7} and Lemma \ref{lem8}, the estimate \eqref{inegprop2-3} yields our final 
estimate \eqref{estimprop2} as stated in Proposition \ref{prop2}.
\end{proof}

The proof of Theorem \ref{thm1} is now a mere application of the above results. Given $u^n \in \mathcal{H}$, the sequence $u^{n+1}$ is defined 
on the interior set of indices by \eqref{LW}. With our now usual notation for $v^n$ and $w^n$, we thus get (compare with \eqref{energie-cauchy} 
when there is no boundary):
\begin{align*}
\| \, u^{n+1} \, \|_{\ell^2(\mathcal{I})}^2 \, - \, \| \, u^n \, \|_{\ell^2(\mathcal{I})}^2 \, =& \, 
\| \, w^n \, \|_{\ell^2(\mathcal{I})}^2 \, + \, \| \, v^n \, \|_{\ell^2(\mathcal{I})}^2 \, - \, 2 \, \langle u^n \, ; \, w^n \rangle_{\ell^2(\mathcal{I})} \\
& \, + \, 2 \, \langle u^n \, ; \, v^n \rangle_{\ell^2(\mathcal{I})} \, - \, 2 \, \langle v^n \, ; \, w^n \rangle_{\ell^2(\mathcal{I})} \, .
\end{align*}
Combining Lemma \ref{lem6} with Proposition \ref{prop2}, we get:
$$
\| \, w^n \, \|_{\ell^2(\mathcal{I})}^2 \, + \, \| \, v^n \, \|_{\ell^2(\mathcal{I})}^2 \, - \, 2 \, \langle u^n \, ; \, w^n \rangle_{\ell^2(\mathcal{I})} 
\, \le \, - \, (\mu \, b)^2 \, \dfrac{(\lambda \, a)^2 +(\mu \, b)^2}{8} \, \sum_{k \in \Z} \, (\Delta_2 \, u_{0,k}^n)^2 \, ,
$$
as long as the parameters $\lambda$ and $\mu$ satisfy the stability condition \eqref{CFLlaxwendroff}. With Corollary \ref{cor2}, we then get the 
result of Theorem \ref{thm1} (recalling the definition \eqref{normH} for the norm of the Hilbert space $\mathcal{H}$).

\section{The problem in a quarter-space}

\subsection{The main result}
\label{section4}

To conclude this article we show that the energy method described in Sections \ref{section2} and \ref{section3} also applies to the quarter space 
geometry at least for one simple and natural \textit{corner extrapolation} condition. We consider the following outgoing (for both sides of the boundary) 
transport equation:
\begin{equation}\label{transport-quart}
\begin{cases}
\partial_t u \, + \, a \, \partial_x u \, + \, b \, \partial_y u \, = \, 0 \, ,& t \ge 0 \, ,\, (x,y) \in \R^+ \times \R^+ \, ,\\
u_{|_{t=0}} \, = \, u_0 \, ,& 
\end{cases}
\end{equation}
where we assume $a \, , \, b <0$, so that no boundary conditions are required on the two sides of the boundary $\{ x\,=\,0 \}$ and $\{ y\,=\,0 \}$ of the 
space domain.

Let $j \in \Z$; the numerical scheme \eqref{LW} requires the knowledge of the $u_{j-1,k+1}^n$, $u_{j-1,k}^n$, $u_{j-1,k-1}^n$ to determine $u_{j,k}^{n+1}$. 
More precisely, we shall keep in mind that the coefficient of $u_{j-1,k-1}^n$ in the expression of $u_{j,k}^{n+1}$ in \eqref{LW} equals $-(\lambda \, a-\mu 
\, b)^2/8$. For the quarter space problem \eqref{transport-quart}, we use the discrete set of indices $\mathfrak{I} := \N \times \N$ for the \emph{interior} 
values of the numerical solution, so that, unless we assume $\lambda \, a=\mu \, b$, which is very restrictive, the determination of the interior value 
$u_{0,0}^{n+1}$ requires prescribing some numerical boundary condition for the \emph{corner} value $u_{-1,-1}^n$. We also use below the notation 
$\mathfrak{J} := (\{ -1 \} \cup \N)^2$ for the full set of indices corresponding to the cells on which the numerical solution is defined (including 
the so-called \emph{ghost cells} that corresponds to $j=-1$ or $k=-1$ in our notation). There are actually two types of \emph{ghost cells}: the 
boundary ones, which correspond to $j=-1$ and $k\geq 0$, or $j\geq 0$ and $k=-1$, and the corner \textit{ghost cell} corresponding to $j=k=-1$. 
Consequently some numerical boundary conditions are required both for the boundary ghost cells and for the corner ghost cell. The grid is depicted 
in Figure \ref{fig:maillage2} with the interior cells in blue, the boundary ghost cells in red and the corner ghost cell in green.

\begin{figure}[htbp]
\begin{center}
\begin{tikzpicture}[scale=1.5,>=latex]
\draw [ultra thin, dotted, fill=blue!20] (-3,-1.5) rectangle (3.5,2);
\draw [ultra thin, dotted, fill=red!20] (-3.5,-1.5) rectangle (-3,2);
\draw [ultra thin, dotted, fill=red!20] (-3,-2) rectangle (3.5,-1.5);
\draw [ultra thin, dotted, fill=green!20] (-3.5,-2) rectangle (-3,-1.5);
\draw [thin, dashed] (-3.5,-2) grid [step=0.5] (3.5,2);
\draw[thick,black,->] (-4,-1.5) -- (4,-1.5) node[below] {$x$};
\draw[thick,black,->] (-3,-2.2)--(-3,2.5) node[right] {$y$};
\draw (-3.7,-2.12) node[right]{$-\Delta y$};
\draw (-3.5,-0.85) node[right]{$\Delta y$};
\draw (-3.7,1.15) node[right, fill=red!20]{$k \, \Delta y$};
\draw [thin, dashed] (-3.5,1) grid [step=0.5] (-3,1.5);
\draw (-3.6,-1.68) node{$-\Delta x$};
\draw (-3.18,-1.68) node[right]{$\, \, \, 0$};
\draw (-2.6,-1.68) node{$\, \, \, \, \, \Delta x$};
\draw (2.35,-1.68) node{$\, \, \, \, \, j \, \Delta x$};
\node (centre) at (-3.5,-1.5){$\times$};
\node (centre) at (-3,-2){$\times$};
\node (centre) at (-3,-1.5){$\times$};
\node (centre) at (-2.5,-1.5){$\times$};
\node (centre) at (2.5,-1.5){$\times$};
\node (centre) at (-3,1){$\times$};
\node (centre) at (-3,-1){$\times$};
\draw (2.75,1.25) node{$u_{j,k}^n$};
\end{tikzpicture}
\caption{The grid for the quarter-space problem. Interior cells appear in blue, the boundary (ghost) cells appear in red and the corner cell appears in green. 
The value $u_{j,k}^n$ corresponds to the approximation in the cell $[j \, \Delta x,(j+1) \, \Delta x) \times [k \, \Delta y,(k+1) \, \Delta y)$.}
\label{fig:maillage2}
\end{center}
\end{figure}
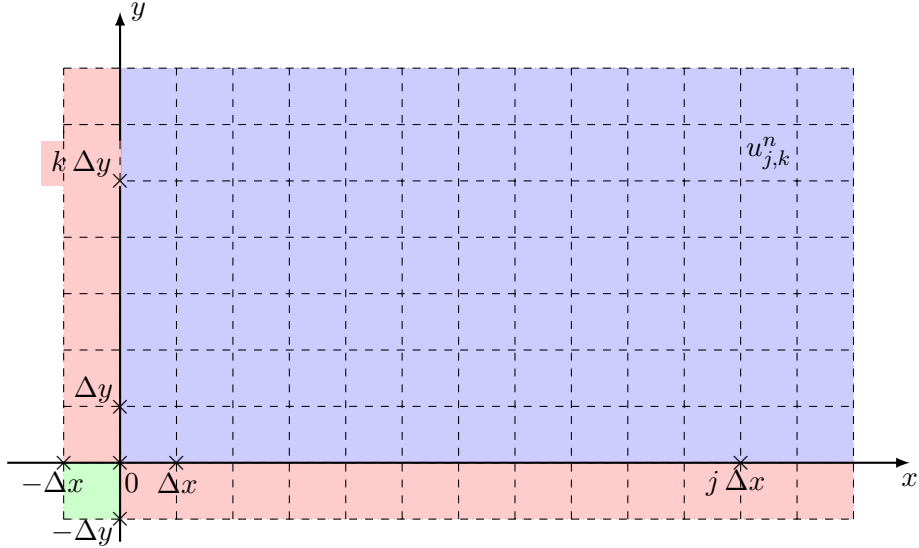

Following the analysis made in the half-space for boundary ghost cells, we will impose extrapolation boundary conditions:
\begin{subequations}
\begin{align}
\label{extrapolationk}
\forall \, n \in \N \, ,\quad \forall \, k \in \N \, ,\quad u_{-1,k}^n \, &= \, u_{0,k}^n \, ,\\
\label{extrapolationj}
\forall \, n \in \N \, ,\quad \forall \, j \in \N \, ,\quad u_{j,-1}^n \, &= \, u_{j,0}^n \, ,
\end{align}
\end{subequations}
in conjunction with the numerical scheme \eqref{LW} for $(j,k) \in \mathfrak{I}$ (that is, for interior values).

Finally because the determination of $u_{0,0}^{n+1}$ requires the value of $u^n_{-1,-1}$ we close the scheme with the prescription of a corner condition. 
Let us discuss a little the choice of such condition. Because \eqref{extrapolationj} and \eqref{extrapolationk} have stencil one, it seems natural to use a 
condition of stencil one also for the corner condition. So that we shall determine $u_{-1,-1}^n$ from the value of one of his three neighbours $u_{-1,0}^n$, 
$u_{0,0}^n$ or $u_{0,-1}^n$. However because of the boundary conditions \eqref{extrapolationj} and \eqref{extrapolationk} these three terms are equal. 
Consequently we impose for the corner condition that 
\begin{align}
\label{corner}
\forall \, n \in \N \, \quad u_{-1,-1}^n \, &= \, u_{0,0}^n \, ,
\end{align}
which is (formally) consistent for smooth solutions to the transport equation.

Other corner conditions such as $u_{-1,-1}^n \, = \, \delta\,u_{0,0}^n$, $\delta \in \mathbb{R}$, or with wider stencils will be discussed in a forthcoming 
contribution. Let us however point that the numerical results of Section \ref{part-num} meet with the intuition that a large choice of $\delta$ leads to 
instability of the associated scheme.\\

Like for the half-space, it is useful below to define the Hilbert space $\mathfrak{H}$ of real valued, square integrable sequences on $\mathfrak{J}$ that 
satisfy the extrapolation boundary conditions \eqref{extrapolationk} and \eqref{extrapolationj}, and the corner condition \eqref{corner} that is:
\begin{equation}
\nonumber
\mathfrak{H} \, := \, \Big\{ u \in \ell^2(\mathfrak{J};\R) \quad |\quad \forall \, k \in \N\cup\lbrace -1\rbrace \, ,\quad u_{-1,k} \, = \, u_{0,k} 
\quad \text{ and } \quad \forall \, j \in \N \, ,\quad u_{j,-1} \, = \, u_{j,0} \Big\} \, .
\end{equation}
Once again the norm on $\mathfrak{H}$ is defined by using only the interior values of $u$:
\begin{equation}
\nonumber
\forall \, u \in \mathfrak{H} \, ,\quad \| \, u \, \|_{\mathfrak{H}}^2 \, := \, \sum_{(j,k) \in \mathfrak{I}} \, u_{j,k}^2 \, .
\end{equation}
In this setting, our main result is the following.

\begin{theorem}
\label{thm2}
Let $a\,<\,0$ and $b \,<\,0$, and let the parameters $\lambda$, $\mu$ satisfy the stability condition \eqref{CFLlaxwendroff}. Then the numerical scheme 
consisting of \eqref{LW} on $\mathfrak{I}$ with the extrapolation numerical boundary condition \eqref{extrapolationk}, \eqref{extrapolationj}, the corner 
condition \eqref{corner} and an initial condition $u^0 \in \mathfrak{H}$ satisfies the following property: for any $n \in \N$, there holds:
\begin{align}
\|\, u^{n+1} \, \|_{\mathfrak{H}}^2 \, -& \, \|\, u^n \, \|_{\mathfrak{H}}^2 \, + \, \frac{\lambda \, | \, a \, |}{8} \, \sum_{k \geq 0} \, (u_{0,k}^n)^2 \, + \, 
\frac{\mu \, | \, b \, |}{8}\, \, \sum_{j \geq 0} \, (u_{j,0}^n)^2 \notag \\
&+ \,\frac{(\mu \, b)^2}{32}  \,\left((\lambda \, a)^2 + (\mu \, b)^2\right) \, \sum_{k \geq 0} \, (\Delta_2 \, u_{0,k}^n)^2 \, + \, 
\frac{(\lambda \, a)^2}{32} \,\left((\lambda \, a)^2 + (\mu \, b)^2\right) \, \sum_{j \geq 0} \, (\Delta_1 \, u_{j,0}^n)^2\, \le \, 0 \, .\label{estimthm2}
\end{align}
In particular, there holds:
\begin{multline*}
\sup_{n \in \N} \, \| \, u^n \, \|_{\mathfrak{H}}^2 \, + \, \frac{\,\lambda \, | \, a \, |}{8} \, \sum_{n \in \N} \, \sum_{k \geq 0} \, (u_{0,k}^n)^2 \, + \, 
\frac{\,\mu \, | \, b \, |}{8} \, \sum_{n \in \N} \, \sum_{j \geq 0} \, (u_{j,0}^n)^2 \\
+ \,\frac{(\mu \, b)^2}{32}  \,\left((\lambda \, a)^2 + (\mu \, b)^2\right) \, \sum_{n \in \N} \, \sum_{k \geq 0} \, (\Delta_2 \, u_{0,k}^n)^2 \, + \, 
\frac{(\lambda \, a)^2}{32} \,\left((\lambda \, a)^2 + (\mu \, b)^2\right) \, \sum_{n \in \N} \, \sum_{j \geq 0} \, (\Delta_1 \, u_{j,0}^n)^2 \\ 
\le \, 2 \, \, \|\, u^0 \, \|_{\mathfrak{H}}^2 \, .
\end{multline*}
\end{theorem}

Theorem \ref{thm2} shows that the corner condition \eqref{corner} maintains stability for the numerical scheme \eqref{LW} when used in combination with 
the Neumann boundary condition on each side of the boundary. Based on this stability estimate, it is very likely that the numerical scheme \eqref{LW} with 
the extrapolation conditions \eqref{extrapolationk}, \eqref{extrapolationj} and \eqref{corner} on the boundary is convergent for smooth enough initial data. 
However, since the numerical boundary conditions will produce first order consistency errors, the second order accuracy of the Lax-Wendroff scheme 
\eqref{LW} will be deteriorated. We shall explore the construction of second order numerical boundary conditions in a future work and thus postpone 
a complete convergence result to this setting.

\subsection{Proof of Theorem \ref{thm2}}

The proof of Theorem \ref{thm2} follows closely the one in the half-space exposed in Paragraph \ref{paragraphpreuve} except that we will have to deal with 
more boundary terms. However let us point that the choice of the particular corner condition \eqref{corner} simplifies a lot the analysis of such terms since 
many will actually vanish. This would no longer be true for more elaborate corner conditions.

The three main ingredients of the proof are the analogues of Lemmas \ref{lem4} and \ref{lem6} and of Proposition \ref{prop2} that are stated just below.

\begin{lemma}
\label{lem9}
Let $a<0$ and $b < 0$. Let $u^n \in \mathfrak{H}$, and let the sequences $v^n,w^n$ be defined on the set of interior indices $\mathfrak{I}$ by \eqref{defvwn'}. 
Then there holds:
\begin{subequations}
\label{termes-antisym9}
\begin{align}
2\, \langle u^n\, ;\, v^n\rangle_{\ell^2(\mathfrak{I})} \,= \, & \, - \, \lambda\,\vert\, a\,\vert \, \sum_{k\geq 0} \, (u_{0,k}^n)^2 \, - \, 
\mu\,\vert \, b\,\vert \, \sum_{j\geq 0} \, (u_{j,0}^n)^2 \, ,\label{lem9-antisym1} \\
\label{lem9-antisym2}
2 \, \langle v^n \, ; \, w^n \rangle_{\ell^2(\mathfrak{I})} \, = \, & \,-\,\frac{\lambda\,\vert \, a\vert \, (\mu \, b)^2}{2}\,\sum_{k\geq 0}\, u_{0,k}^n \, \Delta_2\, u_{0,k+1}^n 
\, - \, \frac{(\lambda\, a)^2\,\mu\,\vert\,b\, \vert}{2}\,\sum_{j\geq 0}\, u_{j,0}^n \, \Delta_1\, u_{j+1,0}^n \\
& \,-\,\frac{\lambda\,\vert \, a\vert \, (\mu \, b)^2}{2} \, u_{0,0}^n \, \Delta_2\, u_{0,0}^n \, - \, 
\frac{(\lambda\, a)^2\,\mu\,\vert\,b\, \vert}{2} \, u_{0,0}^n \, \Delta_1\, u_{0,0}^n \, .\notag
\end{align}
\end{subequations}
\end{lemma}

\begin{lemma}
\label{lem10}
Let $a<0$ and $b < 0$. Let $u^n \in \mathfrak{H}$, and let the sequences $v^n,w^n \in \ell^2(\mathfrak{I};\R)$ be defined by \eqref{defvwn'}. Then there holds:
\begin{align*}
\| \, v^n \, \|_{\ell^2(\mathfrak{I})}^2 \, - \, 2 \, \langle u^n \, ; \, w^n \rangle_{\ell^2(\mathfrak{I})} \, = \, 
& \, - \, \dfrac{(\lambda \, a)^2}{4} \, \| \, \Delta_1 \, u^n \, \|_{\ell^2(\mathfrak{I})}^2 \, - \, \dfrac{(\mu \, b)^2}{4} \, \| \, \Delta_2 \, u^n \, \|_{\ell^2(\mathfrak{I})}^2 \\
& \, - \, \dfrac{(\lambda \, a)^2 +(\mu \, b)^2}{4} \, \| \, D_{1,+} \, D_{2,+} \, u^n \, \|_{\ell^2(\mathfrak{I})}^2 \, + \, 
\lambda\, \vert \, a\,\vert \, \mu\, \vert\, b\,\vert \, (u_{0,0}^n)^2\, .
\end{align*}
\end{lemma}

\begin{proposition}
\label{prop3}
Let $u^n \in \mathfrak{H}$, and let the sequence $w^n \in \ell^2(\mathfrak{I};\R)$ be defined by \eqref{defwjkn'}. Then there holds:
\begin{align}
\label{estimprop3}
4 \, \| \, w^n \, \|_{\ell^2(\mathfrak{I})}^2 \, + \, & \, (\mu \, b)^2 \, \dfrac{(\lambda \, a)^2 +(\mu \, b)^2}{2} \, \sum_{k\geq 0} \, (\Delta_2 \, u_{0,k}^n)^2 
\, + \, (\lambda \, a)^2 \, \dfrac{(\lambda \, a)^2 +(\mu \, b)^2}{2} \, \sum_{j\geq 0} \, (\Delta_1 \, u_{j,0}^n)^2 \\
\nonumber \le \, & \, 2 \, \big( (\lambda \, a)^2 +(\mu \, b)^2 \big) \, \Big( (\lambda \, a)^2 \, \| \, \Delta_1 \, u^n \, \|_{\ell^2(\mathfrak{I})}^2 
+ \, (\mu \, b)^2 \, \| \, \Delta_2 \, u^n \, \|_{\ell^2(\mathfrak{I})}^2 \\
\nonumber
\, & \qquad \qquad \qquad \qquad + \, \big( (\lambda \, a)^2 +(\mu \, b)^2 \big) \, \| \, D_{1,+} \, D_{2,+} \, u^n \, \|_{\ell^2(\mathfrak{I})}^2 \Big) \, .
\end{align}
\end{proposition}

In all what follows, we omit the dependency with respect to $n$ in the notation and we use the short hand notation $\alpha \,:=\,\lambda\, a$, 
$\beta\,:=\,\mu\, b$. As in the half-space geometry, the following technical lemma will be particularly helpful.

\begin{lemma}
\label{lemaux}
Let $U,\,V \in \mathfrak{H}$. Then for $p\in \lbrace 1,2\rbrace$ there hold the relations:
\begin{subequations}
\label{relations-lemaux}
\begin{align}
\langle U \, ; \, \Delta_p \, V \rangle_{\ell^2(\mathfrak{I})} \, = \, & \, \langle \Delta_p \, U \, ; \, V \rangle_{\ell^2(\mathfrak{I})} \, ,\label{lemaux-1} \\
\langle D_{p,0} \, U \, ; \, \Delta_p \, U \rangle_{\ell^2(\mathfrak{I})} \, = \, & \, 0 \, ,\label{lemaux-2} \\
\langle U \, ; \, \Delta_p \, U \rangle_{\ell^2(\mathfrak{I})} \, = \, & \, - \, \| \, D_{p,+} \, U \, \|_{\ell^2(\mathfrak{I})}^2 \, ,\label{lemaux-3} \\
\|\, D_{p,0} \, U \, \|_{\ell^2(\mathfrak{I})} \, = \, & \, \| \, D_{p,+} \, U \, \|_{\ell^2(\mathfrak{I})}^2 \, - \, \dfrac{1}{4} \, 
\| \, \Delta_p \, U \, \|_{\ell^2(\mathfrak{I})}^2 \, ,\label{lemaux-4}
\end{align}
\end{subequations}
\end{lemma}

The proof is a straightforward modification of the one of Lemma \ref{lem5}. The sums now run for $k\in\mathbb{N}$ and not $k\in \mathbb{Z}$, so 
that we feel free to omit the proof and leave it to the interested reader. Note in particular that for Lemma \ref{lemaux} to hold we only require the 
boundary conditions \eqref{extrapolationk} and \eqref{extrapolationj} to hold. The relations \eqref{relations-lemaux} are independent of the corner 
condition \eqref{corner}. We now turn to the proofs of the three above results that will lead us to Theorem \ref{thm2}.

\begin{proof}[Proof of Lemma \ref{lem9}]
We start with the proof of \eqref{lem9-antisym1}. Reiterating the same computations as in the half-space gives:
\[ \forall \, (j,k)\in \mathfrak{I} \, ,\quad 
2\,u_{j,k}\,v_{j,k}\,=\,-\,\alpha \, D_{1,+}\,\left\lbrace u_{j-1,k}\,u_{j,k}\right\rbrace\,-\,\beta \, D_{2,+}\,\left\lbrace u_{j,k-1}\,u_{j,k}\right\rbrace.
\]
We sum with respect to $(j,k)\in\mathfrak{I}$ and use the boundary conditions \eqref{extrapolationk} and \eqref{extrapolationj} to end up with 
\[
2\, \langle u\, ;\, v\rangle_{\ell^2(\mathfrak{I})}\,=\, \alpha \, \sum_{k\geq 0} \, (u_{0,k})^2 \,+\, \beta \, \sum_{j\geq 0} \, (u_{j,0})^2 \, ,
\]
that is to say \eqref{lem9-antisym1} (since $\alpha$ and $\beta$ are negative).

We now turn to the proof of \eqref{lem9-antisym2}. For convenience let us recall that
\begin{align}
2 \, \langle v \, ; \, w \rangle_{\ell^2(\mathfrak{I})} \, = \, & \, \alpha^3 \, {\color{ForestGreen} \langle D_{1,0} \, u \, ; \, \Delta_1 \, u \rangle_{\ell^2(\mathfrak{I})}} 
\, + \, \beta^3 \, {\color{ForestGreen} \langle D_{2,0} \, u \, ; \, \Delta_2 \, u \rangle_{\ell^2(\mathfrak{I})}} \notag \\
&\, + \, \alpha \, \beta^2 \, {\color{red}\langle D_{1,0} \, u \, ; \, \Delta_2 \, u \rangle_{\ell^2(\mathfrak{I})}} 
\, + \, \alpha^2 \, \beta \, {\color{red}\langle D_{2,0} \, u \, ; \, \Delta_1 \, u \rangle_{\ell^2(\mathfrak{I})}} \label{decompositionvw-1-bis} \\
&\, + \, 2 \, \alpha \, \beta^2 \,{\color{blue} \langle D_{2,0} \, u \, ; \, D_{1,0} \, D_{2,0} \, u \rangle_{\ell^2(\mathfrak{I})}} 
\, + \, 2 \, \alpha^2 \, \beta \, {\color{blue} \langle D_{1,0} \, u \, ; \, D_{2,0} \, D_{1,0} \, u \rangle_{\ell^2(\mathfrak{I})}} \notag \\
&\, - \, \dfrac{\alpha^2 + \beta^2}{4} \, \langle \Delta_1 \, \Delta_2 \, u \, ; \, \alpha \, D_{1,0} \, u \, + \, \beta \, D_{2,0} \, u \rangle_{\ell^2(\mathfrak{I})} \, .\notag
\end{align}
From equation \eqref{lemaux-2}, the two green terms in the first line of \eqref{decompositionvw-1-bis} vanish. We then consider the two blue terms in the 
third line. In order to do so we recall that from \eqref{expressio-ipp} we have\footnote{The crucial observation here is that $V$ is not necessarily defined on 
the whole set $\mathfrak{J}$ and also needs not satisfy the extrapolation boundary conditions.}
\begin{equation}
\label{D10}
\begin{split}
\forall \, V \in \ell^2((\{ -1 \} \cup \N) \times \N) \, ,\quad \langle V \, ; \, D_{1,0} \, V \rangle\,=\,-\,\frac{1}{2} \, \sum_{k\geq 0} \, V_{-1,k}\,V_{0,k} \\
\forall \, V \in \ell^2(\N \times (\{ -1 \} \cup \N)) \, ,\quad \langle V\,;\, D_{2,0}\,V\rangle\,=\,-\,\frac{1}{2} \, \sum_{j\geq 0}V_{j,-1}\,V_{j,0} \, .
\end{split}
\end{equation}
We use \eqref{D10} for the sequence $D_{2,0}\,u$ to obtain 
\begin{align*}
\langle D_{2,0}\,u\,;\,D_{1,0}\,D_{2,0}\,u\rangle\,=&\,-\,\dfrac{1}{8} \, \sum_{k\geq 0}(u_{-1,k+1}-u_{-1,k-1})\,(u_{0,k+1}-u_{0,k-1})\\
=&\,-\,\dfrac{1}{8} \, \sum_{k\geq 1}(\underbrace{u_{-1,k+1}-u_{-1,k-1})}_{= \, u_{0,k+1}-u_{0,k-1}}\,(u_{0,k+1}-u_{0,k-1})
-\dfrac{1}{8} \, \underbrace{(u_{-1,1}-u_{-1,-1})}_{= \, u_{0,1}-u_{0,0}}\,(u_{0,1}-u_{0,-1}),
\end{align*}
where we used the boundary condition \eqref{extrapolationk} and the corner condition \eqref{corner}. Because of the boundary condition \eqref{extrapolationk} 
we have $u_{0,0}=u_{0,-1}$ so that 
\begin{equation}\nonumber
\langle D_{2,0}\,u\,;\,D_{1,0}\,D_{2,0}\,u\rangle\,=\,-\,\frac{1}{2} \, \sum_{k\geq 0} \, (D_{2,0} \, u_{0,k})^2 \, .
\end{equation}
Proceeding similarly gives
\begin{equation}\nonumber
\langle D_{1,0}\,u\,;\,D_{2,0}\,D_{1,0}\,u\rangle\,=\,-\,\frac{1}{2} \, \sum_{j\geq 0} \, (D_{1,0} \, u_{j,0})^2 \, ,
\end{equation}
and we have thus obtained the expressions of the blue terms in the right hand side of \eqref{decompositionvw-1-bis}.

To deal with the two reds terms in the right hand side of \eqref{decompositionvw-1-bis} we first write $\Delta_2\,=\,D_{2,-}\,D_{2,+}$ (instead of 
$\Delta_2\,=\,-(D_{2,-}^*)\,D_{2,+}$ in the half-space) and we distribute $D_{2,-}$ thanks to the following discrete integration by parts formula\footnote{The 
important point here is that \eqref{ipp-simple1} is valid for sequences $V$ that are defined on $\mathfrak{I}$ and not on any ghost cell.} for $V \in \ell^2 
(\mathfrak{I};\mathbb{R})$ and either $W \in \ell^2((\{ -1 \} \cup \N)\times \N;\R)$ or $W \in \ell^2(\N \times (\{ -1 \} \cup \N);\R)$:
\begin{subequations}
\begin{align}
\label{ipp-simple1}
\langle D_{1,+}\, V \,; \, W\rangle_{\ell^2(\mathfrak{I})}\,&=\, -\sum_{k\geq 0} \, V_{0,k}\,W_{-1,k}\,-\,\langle  V\,;\, D_{1,-}\, W\rangle_{\ell^2(\mathfrak{I})} \, ,\\
\label{ipp-simple2}
\langle D_{2,+}\, V \,; \, W\rangle_{\ell^2(\mathfrak{I})}\,&=\, -\sum_{j\geq 0} \, V_{j,0}\,W_{j,-1}\,-\,\langle  V\,;\, D_{2,-}\, W\rangle_{\ell^2(\mathfrak{I})} \, .
\end{align}
\end{subequations}
We therefore have from the boundary conditions \eqref{extrapolationk} and \eqref{extrapolationj}:
\begin{align*}
\langle D_{1,0} \, u \, ; \, \Delta_2 \, u \rangle_{\ell^2(\mathfrak{I})}\,=\, -\langle D_{1,0}\,D_{2,+} \, u \, ; \, D_{2,+} \, u \rangle_{\ell^2(\mathfrak{I})} 
\, - \, \sum_{j\geq 0} \, (D_{1,0} \, u)_{j,0} \, \underbrace{D_{2,+} \, u_{j,-1}}_{=\, 0} \, = \, \dfrac{1}{2} \, \sum_{k\geq 0} \, (D_{2,+}\, u_{0,k})^2 \, ,
\end{align*}
where we used \eqref{D10} for the sequence $D_{2,+}\,u$. Similarly, we get:
\begin{align*}
\langle D_{2,0} \, u \, ; \, \Delta_1 \, u \rangle_{\ell^2(\mathfrak{I})}\,=\,\dfrac{1}{2} \, \sum_{j\geq 0} \, (D_{1,+}\, u_{j,0})^2 \, .
\end{align*}
So that using the same computations as the ones in the proof of Lemma \ref{lem4} gives
\begin{align}\nonumber
2 \, \langle v \, ; \, w \rangle_{\ell^2(\mathfrak{I})} \, = \, & \,-\dfrac{\alpha\,\beta^2}{2}\,\sum_{k\geq 0} \, D_{2,-}\,u_{0,k}\,D_{2,+}\, u_{0,k} 
\, - \, \dfrac{\alpha^2\,\beta}{2} \, \sum_{j\geq 0} \, D_{1,-}\,u_{j,0}\,D_{1,+}\, u_{j,0} \\
\label{eq-avant-derniere} \, &- \, 
\dfrac{\alpha^2 + \beta^2}{4} \, {\color{orange}\langle \Delta_1 \, \Delta_2 \, u \, ; \, \alpha \, D_{1,0} \, u \, + \, \beta \, D_{2,0} \, u \rangle_{\ell^2(\mathfrak{I})}} \, .
\end{align}

We shall now justify that the orange term in \eqref{eq-avant-derniere} vanishes. This is done essentially like in the half-space by first writing 
$\Delta_2\,=\,D_{2,-}\,D_{2,+}$ and then integrating by parts in order to use the formula \eqref{lemaux-2}. Indeed, we compute:
\begin{align*}
\langle \Delta_1\,\Delta_2\, u\, ;\, D_{1,0}\, u\rangle_{\ell^2(\mathfrak{I})}\, & \, =\, 
\langle \Delta_1\,D_{2,+}\, u\, ;\, D_{1,0}\, D_{2,+}\, u\rangle_{\ell^2(\mathfrak{I})} \\
& \, = \, \dfrac{1}{2}\,\sum_{k\geq 0} \, (D_{2,+}\,u_{0,k}\,-\,D_{2,+}\, u_{-1,k})^2 \, = \, 0 \, ,
\end{align*}
because of the boundary condition \eqref{extrapolationk} which implies that for all $k\in\mathbb{N}$, $D_{2,+}\,u_{0,k}\,=\,D_{2,+}\, u_{-1,k}$.

To complete the proof of equation \eqref{lem9-antisym2} we should modify a little the expression in the right hand side of \eqref{eq-avant-derniere}. 
We use the discrete integration by parts formulas \eqref{ipp-simple1} and \eqref{ipp-simple2} (for $j,k\geq 1$) in order to express for instance
\begin{align*}
\sum_{k\geq 0}\,D_{2,-}\,u_{0,k}\,D_{2,+}\, u_{0,k}\,= \, &\,-\,\sum_{k\geq 0} \, u_{0,k}\,D^2_{2,+}\,u_{0,k} \, - \, u_{0,-1} \, D_{2,+} \, u_{0,0} \\
= \, &\,-\,\sum_{k\geq 0} \, u_{0,k}\,\Delta_2\,u_{0,k+1} \, - \, u_{0,0} \, \Delta_2 \, u_{0,0} \, ,
\end{align*}
where we used the fact that $D_{2,-}\,u_{0,0}\,=\,0$ because of \eqref{extrapolationk}. The sum with respect to $j$ in \eqref{eq-avant-derniere} is dealt 
with similarly and the proof of Lemma \ref{lem9} is now complete.
\end{proof}

\begin{proof}[Proof of Lemma \ref{lem10}]
For convenience let us recall that we have the expansion 
\begin{align*}
\| \, v \, \|_{\ell^2(\mathfrak{I})}^2 \, - \, 2 \, \langle u \, ; \, w \rangle_{\ell^2(\mathfrak{I})} \, =& \, \, 
\alpha^2 \, \| \, D_{1,0} \, u \, \|_{\ell^2(\mathfrak{I})}^2 \, + \, \beta^2 \, \| \, D_{2,0} \, u \, \|_{\ell^2(\mathfrak{I})}^2 \\
&+ \, 2 \, \alpha \, \beta \, \langle D_{1,0 } \, u \, ; \, D_{2,0} \, u \rangle_{\ell^2(\mathfrak{I})} \, + \, 
2 \, \alpha \, \beta \, \langle u \, ; \, D_{1,0 } \, D_{2,0} \, u \rangle_{\ell^2(\mathfrak{I})} \\
&+ \, \alpha^2 \, \langle u \, ; \, \Delta_1 \, u \rangle_{\ell^2(\mathfrak{I})} \, + \, \beta^2 \, \langle u \, ; \, \Delta_2 \, u \rangle_{\ell^2(\mathfrak{I})} 
\, - \, \dfrac{\alpha^2 + \beta^2}{4} \, \langle u \, ; \, \Delta_1 \, \Delta_2 \, u \rangle_{\ell^2(\mathfrak{I})} \, .
\end{align*}
It is a mere exercise to generalize the integration by parts formula \eqref{expressio-ipp} to any couple of sequences in the quarter-space as follows:
\[
\langle D_{1,0}\,V\, ;\, W\rangle_{\ell^2(\mathfrak{I})}\,=\,-\langle V\, ;\,D_{1,0}\, W\rangle_{\ell^2(\mathfrak{I})} \,-\, 
\dfrac{1}{2}\, \sum_{k\geq 0} \, V_{-1,k}\,W_{0,k}\,-\,\dfrac{1}{2}\, \sum_{k\geq 0} \, V_{0,k}\,W_{-1,k} \, .
\]
Applying this identity to $V\,=\,u$ and $W\,=\, D_{2,0}\,u$ with $u \in \mathfrak{H}$ gives
\begin{align*}
2\,\alpha\,\beta\,\langle D_{1,0}\, u\, ;\, D_{2,0}\, u\rangle_{\ell^2(\mathfrak{I})}\,=&\,-\,2\,\alpha\,\beta\,\langle u\, ;\,D_{1,0}\, D_{2,0}\, u\rangle_{\ell^2(\mathfrak{I})} \\
&\,-\,\alpha\,\beta\, \sum_{k\geq 0} \, u_{-1,k}\,D_{2,0}u_{0,k}\,-\,\alpha \, \beta\, \sum_{k\geq 0} \, u_{0,k}\,D_{2,0}u_{-1,k} \\
:=&\,-2\,\alpha\,\beta\,\langle u\, ;\,D_{1,0}\, D_{2,0}\, u\rangle_{\ell^2(\mathfrak{I})} \, - \, \mathcal{B} \, .
\end{align*}
We make the boundary term $\mathcal{B}$ explicit by using the definition of the centered difference $D_{2,0}$ and the boundary conditions \eqref{extrapolationk} 
and \eqref{corner}:
\begin{align*}
\mathcal{B}\,&=\,-\frac{\alpha\, \beta}{2}\, \sum_{k\geq 0} \, \underbrace{u_{-1,k}}_{= \, u_{0,k}}\,(u_{0,k+1} \, - \, u_{0,k-1})\,-\,\frac{\alpha\,\beta}{2}\, 
\sum_{k\geq 0} \, u_{0,k}\,(\underbrace{u_{-1,k+1}}_{=u_{0,k+1}} \, - \, u_{-1,k-1}) \\
&=\,-\,\alpha\,\beta\sum_{k\geq 0} \, u_{0,k}\,u_{0,k+1}\,+\,\alpha\,\beta \,\sum_{k\geq 0} \, u_{0,k}\, u_{0,k-1} \,=\, \alpha\,\beta\, u_{0,0}^2 \, .
\end{align*}
Consequently we obtain
\begin{align*}
\| \, v \, \|_{\ell^2(\mathfrak{I})}^2 \, - \, 2 \, \langle u \, ; \, w \rangle_{\ell^2(\mathfrak{I})} \, =& \, \, 
\alpha^2 \textcolor{blue}{\, \| \, D_{1,0} \, u \, \|_{\ell^2(\mathfrak{I})}^2} \, + \, \beta^2 \, \textcolor{blue}{\| \, D_{2,0} \, u \, \|_{\ell^2(\mathfrak{I})}^2} 
\,+\,\alpha\,\beta\, u_{0,0}^2 \\
&+ \, \alpha^2 \, \textcolor{red}{\langle u \, ; \, \Delta_1 \, u \rangle_{\ell^2(\mathfrak{I})}} \, + \, 
\beta^2 \, \textcolor{red}{\langle u \, ; \, \Delta_2 \, u \rangle_{\ell^2(\mathfrak{I})}} 
\, - \, \dfrac{\alpha^2 + \beta^2}{4} \, \langle u \, ; \, \Delta_1 \, \Delta_2 \, u \rangle_{\ell^2(\mathfrak{I})} \, .
\end{align*}
Then, as in the half-space geometry, we use \eqref{lemaux-3} on the red terms and \eqref{lemaux-4} on the blue terms to obtain
\begin{align*}
\| \, v \, \|_{\ell^2(\mathfrak{I})}^2 \, - \, 2 \, \langle u \, ; \, w \rangle_{\ell^2(\mathfrak{I})} \, =& \, \, 
-\,\frac{\alpha^2}{4}\, \| \, \Delta_1 \, u \, \|_{\ell^2(\mathfrak{I})}^2 \, - \, \frac{\beta^2}{4} \, \| \, \Delta_2 \, u \, \|_{\ell^2(\mathfrak{I})}^2 \,+\,\alpha\,\beta\, u_{0,0}^2 \\
&\, - \, \dfrac{\alpha^2 + \beta^2}{4} \, \langle u \, ; \, \Delta_1 \, \Delta_2 \, u \rangle_{\ell^2(\mathfrak{I})} \, .
\end{align*}

To conclude we write $\Delta_2\,=\,D_{2,+}\,D_{2,-}$ and $\Delta_1\,=\,D_{1,+}\,D_{1,-}$ and we use the discrete integration by parts formulas \eqref{ipp-simple1} 
and \eqref{ipp-simple2} (together with the boundary conditions \eqref{extrapolationk}, \eqref{extrapolationj}, \eqref{corner}), which gives:
\begin{align*}
\langle u \, ; \, \Delta_1 \, \Delta_2 \, u \rangle_{\ell^2(\mathfrak{I})}\,= \, &\,-\langle D_{2,+}\, u \, ; \, \Delta_1 \, D_{2,+} \, u \rangle_{\ell^2(\mathfrak{I})} 
\,-\,\sum_{j\geq 0} \, u_{j,0} \, \underbrace{D_{2,+} \, \Delta_1\, u_{j,-1}}_{= \, 0} \\
=\, &\, \Vert D_{1,+}\,D_{2,+}\, u\Vert^2_{\ell^2(\mathfrak{I})}\,+\,\sum_{k\geq 0} \, D_{2,+}\, u_{0,k}\,\underbrace{D_{1,+} \, D_{2,+} \, u_{-1,k}}_{= \, 0} \, .
\end{align*}
Indeed the first boundary term (in our first integration by parts) vanishes because we can use either \eqref{extrapolationj} or \eqref{corner} to compute:
\begin{equation}\label{equation-annexe}
\Delta_1\,D_{2,+}\, u_{j,-1}\,=\, \Delta_1 \, u_{j,0} \, - \, \Delta_1 u_{j,-1} \, = \, \begin{cases} 0 &\text{if } j\geq 1\text{ from } \eqref{extrapolationj} \, ,\\
0&\text{if } j=0\text{ from } \eqref{corner} \, .
\end{cases}
\end{equation}
Eventually, we end up with 
\begin{align*}
\| \, v \, \|_{\ell^2(\mathfrak{I})}^2 \, - \, 2 \, \langle u \, ; \, w \rangle_{\ell^2(\mathfrak{I})} \, = &\, \, 
-\,\frac{\alpha^2}{4}\, \| \, \Delta_1 \, u \, \|_{\ell^2(\mathfrak{I})}^2 \, - \, \frac{\beta^2}{4} \, \| \, \Delta_2 \, u \, \|_{\ell^2(\mathfrak{I})}^2 \, \\
&\,- \, \dfrac{\alpha^2 + \beta^2}{4} \, \Vert D_{1,+}\,D_{2,+}u\Vert^2_{\ell^2(\mathfrak{I})}\,+\,\alpha\,\beta\, u_{0,0}^2 \, ,
\end{align*}
as desired (recall that $\alpha$ and $\beta$ are both negative).
\end{proof}

\begin{proof}[Proof of Proposition \ref{prop3}]
We follow exactly the same strategy as in the whole space $\Z^2$ and in the half-space. In particular, a crucial first ingredient is to be able to expand the norm 
$\Vert D_{1,0}\,D_{2,0}\, u\Vert^2_{\ell^2(\mathfrak{I})}$ for $u \in \mathfrak{H}$, which is based here on \eqref{lemaux-4}. Namely, we leave to the interested 
reader the verification of the following relation:
\begin{align*}
\forall \, u \in \mathfrak{H} \, ,\quad \Vert D_{1,0}\,D_{2,0}\, u\Vert^2_{\ell^2(\mathfrak{I})} \, = \, & \, \Vert D_{1,+}\,D_{2,+}\, u\Vert^2_{\ell^2(\mathfrak{I})} 
\, + \, \dfrac{1}{16} \, \Vert \Delta_1\,\Delta_2\, u\Vert^2_{\ell^2(\mathfrak{I})} \\
& \, - \, \dfrac{1}{4} \, \Vert D_{1,+}\,\Delta_2\, u\Vert^2_{\ell^2(\mathfrak{I})} \, - \, \dfrac{1}{4} \, \Vert D_{2,+}\,\Delta_1 \, u\Vert^2_{\ell^2(\mathfrak{I})} \, .
\end{align*}
Using then the same crude estimate as in the proof of Proposition \ref{prop1} for the term $\langle \Delta_1 \, u\, ;\, \Delta_2\,u \rangle_{\ell^2(\mathfrak{I})}$ and 
the above equality for the norm $\Vert D_{1,0}\,D_{2,0}\, u\Vert^2_{\ell^2(\mathfrak{I})}$, we are led to the estimate:
\begin{align*}
4 \, \| \, w \, \|_{\ell^2(\mathfrak{I})}^2 \, \le & \, \, (\alpha^2+\beta^2) \, 
\Big( \alpha^2 \, \| \, \Delta_1 \, u \, \|_{\ell^2(\mathfrak{I})}^2 \, + \, \beta^2 \, \| \, \Delta_2 \, u \, \|_{\ell^2(\mathfrak{I})}^2 
\, + \, (\alpha^2+\beta^2) \, \| \, D_{1,+} \, D_{2,+} \, u \, \|_{\ell^2(\mathfrak{I})}^2 \Big) \\
& + \, \dfrac{(\alpha^2+\beta^2)^2}{8} \, \| \, \Delta_1 \, \Delta_2 \, u \, \|_{\ell^2(\mathfrak{I})}^2 
\, - \, \dfrac{(\alpha^2+\beta^2)^2}{4} \, \Big( \| \, D_{1,+} \, \Delta_2 \, u \, \|_{\ell^2(\mathfrak{I})}^2 \, + \, \| \, D_{2,+} \, \Delta_1 \, u \, \|_{\ell^2(\mathfrak{I})}^2 \Big) \\
& + \, 4 \, \alpha \, \beta \, \langle D_{1,0} \, D_{2,0} \, u \, ; \, \alpha^2 \, \Delta_1 \, u \, + \, \beta^2 \, \Delta_2 \, u \rangle_{\ell^2(\mathfrak{I})} 
\, - \,(\alpha^2+\beta^2) \, \alpha \, \beta \, {\color{red} \langle D_{1,0} \, D_{2,0} \, u \, ; \, \Delta_1 \, \Delta_2 \, u \rangle_{\ell^2(\mathfrak{I})}} \\
& - \, \dfrac{\alpha^2+\beta^2}{2} \, {\color{blue} \langle \Delta_1 \, \Delta_2 \, u \, ; \, \alpha^2 \, \Delta_1 \, u \, + \, \beta^2 \, \Delta_2 \, u \rangle_{\ell^2(\mathfrak{I})}} \, ,
\end{align*}
which is the exact analogue of the estimate \eqref{ineg-symw2} that we obtained in the case of the half-space.

For the blue term, we first sum with respect to $k \in \N$ and use the relation $\Delta_2 =D_{2,-} \, D_{2,+}$ and the integration by parts \eqref{ipp-simple2} to get:
\begin{align*}
\langle \Delta_1 \, \Delta_2 \, u \, ;  \, \Delta_1 \, u\rangle_{\ell^2(\mathfrak{I})}\,=\,&\,\langle D_{2,-}\,D_{2,+}\,\Delta_1 \, u \, ; \, \Delta_1 \, u\rangle_{\ell^2(\mathfrak{I})} \\ 
= \, &\,- \, \Vert \Delta_1 \, D_{2,+} u\Vert^2_{\ell^2(\mathfrak{I})}\,-\,\sum_{j\geq 0} \, \underbrace{\Delta_1\,D_{2,+}\, u_{j,-1}}_{= \, 0}\,\Delta_1 u_{j,0} \, .
\end{align*}

For the red term, we use the definition of $D_{1,0}$, the decomposition $\Delta_1\,=\,D_{1,+}\,D_{1,-}\,=\,D_{1,-}\,D_{1,+}$ and the discrete integration formula 
\eqref{ipp-simple1} so that
\begin{align*}
\langle D_{1,0} \, D_{2,0} \, u \, ; \, \Delta_1 \, \Delta_2 \, u &\rangle_{\ell^2(\mathfrak{I})} \\
= \, &\,\frac{1}{2}\,\langle D_{1,+}\, D_{2,0}\,u \, ; \, D_{1,+}\, D_{1,-}\, \Delta_2\, u\rangle_{\ell^2(\mathfrak{I})}\,+\,
\dfrac{1}{2}\,\langle D_{1,-}\, D_{2,0}\,u \, ; \, D_{1,-}\, D_{1,+}\, \Delta_2\, u\rangle_{\ell^2(\mathfrak{I})}\\
= \, &\,-\,\langle \Delta_1\, D_{2,0}\,u\,;\, \Delta_2\,D_{1,0}\, u \rangle_{\ell^2(\mathfrak{I})} \\
&\, -\,\frac{1}{2}\,\sum_{k\geq 0}D_{1,+}\,D_{2,0}\,u_{-1,k}\,\underbrace{\Delta_2\,D_{1,-}\,u_{0,k}}_{=0}\,-\,\frac{1}{2}\,\sum_{k\geq 0}D_{1,-}\,D_{2,0}\,u_{0,k}\,\underbrace{\Delta_2\,D_{1,+}\,u_{-1,k}}_{=0},
\end{align*}
where, in order to justify that the boundary terms vanish, we have used the boundary condition \eqref{extrapolationk} and the corner condition \eqref{corner} and 
the same reasoning as in \eqref{equation-annexe}. We therefore end up with the analogue of \eqref{expressionw}, that is:
\begin{align}
4 \, \| \, w \, \|_{\ell^2(\mathfrak{I})}^2 \, \le & \, \, (\alpha^2+\beta^2) \, 
\Big( \alpha^2 \, \| \, \Delta_1 \, u \, \|_{\ell^2(\mathfrak{I})}^2 \, + \, \beta^2 \, \| \, \Delta_2 \, u \, \|_{\ell^2(\mathfrak{I})}^2 
\, + \, (\alpha^2+\beta^2) \, \| \, D_{1,+} \, D_{2,+} \, u \, \|_{\ell^2(\mathfrak{I})}^2 \Big) \notag \\
& + \, \dfrac{(\alpha^2+\beta^2)^2}{8} \, \| \, \Delta_1 \, \Delta_2 \, u \, \|_{\ell^2(\mathfrak{I})}^2 
\, + \, 4 \, \alpha \, \beta \, \langle D_{1,0} \, D_{2,0} \, u \, ; \, \alpha^2 \, \Delta_1 \, u \, + \, \beta^2 \, \Delta_2 \, u \rangle_{\ell^2(\mathfrak{I})} \label{prop44-expression} \\
& + \, {\color{blue} \dfrac{(\alpha^2+\beta^2)}{4} \, (\alpha^2-\beta^2) \, \Big( 
\| \, D_{2,+} \, \Delta_1 \, u \, \|_{\ell^2(\mathfrak{I})}^2 \, - \, \| \, D_{1,+} \, \Delta_2 \, u \, \|_{\ell^2(\mathfrak{I})}^2 \Big)} \nonumber\\
& + \, {\color{blue} (\alpha^2+\beta^2) \, \alpha \, \beta \, \langle D_{1,0} \, \Delta_2 \, u \, ; \, D_{2,0} \, \Delta_1 \, u \rangle_{\ell^2(\mathfrak{I})}} \, .\notag
\end{align}
Once again, to estimate the blue term, we can follow exactly the same procedure as in the whole space $\Z^2$ (see Step 1 of Proposition \ref{prop1}) or in the 
half-space, except that we sum on $\mathfrak{I}=\N^2$ instead of $\mathbb{Z}^2$ and we use the relation \eqref{lemaux-4} instead of \eqref{formuleslem2}. 
The estimate of the above blue term in \eqref{prop44-expression} gives the (almost final) estimate:
\begin{align}
4 \, \| \, w \, \|_{\ell^2(\mathfrak{I})}^2 \, \le& \, \, (\alpha^2+\beta^2) \, 
\Big( \alpha^2 \, \| \, \Delta_1 \, u \, \|_{\ell^2(\mathfrak{I})}^2 \, + \, \beta^2 \, \| \, \Delta_2 \, u \, \|_{\ell^2(\mathfrak{I})}^2 
\, + \, (\alpha^2+\beta^2) \, \| \, D_{1,+} \, D_{2,+} \, u \, \|_{\ell^2(\mathfrak{I})}^2 \Big) \label{inegprop3-3} \\
& + \, \alpha^2 \, B_1 \, + \, \beta^2 \, B_2 \, .\notag
\end{align}
where we introduced, like for the whole space and half-space problems, the quantities:
\begin{subequations}
\label{defB-corner}
\begin{align}
B_1 \, &:= \, \dfrac{\alpha^2+\beta^2}{2} \, \| \, D_{2,+} \, \Delta_1 \, u \, \|_{\ell^2(\mathfrak{I})}^2 \, + \, 
4 \, \alpha \, \beta \, \langle D_{1,0} \, D_{2,0} \, u \, ; \, \Delta_1 \, u \rangle_{\ell^2(\mathfrak{I})} \, ,\label{defB1-corner} \\
B_2 \, &:= \, \dfrac{\alpha^2+\beta^2}{2} \, \| \, D_{1,+} \, \Delta_2 \, u \, \|_{\ell^2(\mathfrak{I})}^2 \, + \, 
4 \, \alpha \, \beta \, \langle D_{1,0} \, D_{2,0} \, u \, ; \, \Delta_2 \, u \rangle_{\ell^2(\mathfrak{I})} \, .\label{defB2-corner}
\end{align}
\end{subequations}
To conclude we should show the estimates
\begin{subequations}\label{est-finale-prop3}
\begin{align}
\label{est-finale-prop31}
B_1\, \leq \, & \, (\,\alpha^2\,+\,\beta^2\,)\left( \Vert \,\Delta_1\, u \,\Vert^2_{\ell^2(\mathfrak{I})}\,+ \, 
\Vert \, D_{1,+}\,D_{2,+}\, u\,\Vert^2_{\ell^2(\mathfrak{I})}\,\right)\,-\,\frac{\alpha^2\,+\,\beta^2}{2}\,\sum_{j\geq 0}\,(\Delta_1\, u_{j,0})^2 \, ,\\ 
\label{est-finale-prop32}
B_2\, \leq \, &\, (\,\alpha^2\,+\,\beta^2\,)\left( \Vert \,\Delta_2\, u \,\Vert^2_{\ell^2(\mathfrak{I})}\,+\, 
\Vert \, D_{1,+}\,D_{2,+}\, u\,\Vert^2_{\ell^2(\mathfrak{I})}\,\right)\,-\,\frac{\alpha^2\,+\,\beta^2}{2}\,\sum_{k\geq 0}\,(\Delta_2\, u_{0,k})^2 \, .
\end{align}
\end{subequations}
The latter inequalities are direct consequences of Lemma \ref{lem8} (or rather its extension to the quarter-space). Indeed, reiterating the proof of Lemma 
\ref{lem8} directly gives\footnote{Here, unlike in the half-space, the two space coordinates play the same role.} \eqref{est-finale-prop32}. Let us just list the 
main modifications. Because we now sum with respect to $k\in\mathbb{N}$ instead of $k \in \Z$, we have to justify two points:
\begin{itemize}
\item The fact that after using the relations $D_{1,0}\,=\,A_{1,-}\,D_{1,+}$ and $D_{2,0}\,=\,A_{2,-}\,D_{2,+}$, the scalar product in \eqref{defB2-corner} can 
be written in the form $\langle A_{1,-}\,W\,;\,\widetilde{W}\rangle$,  with $W_{-1,k}=0$ for all $k\in \mathbb{N}$. This is still true because of the boundary 
condition \eqref{extrapolationk} and because the relevant sequence $W$ here is $W:=\, A_{2,-}\,D_{1,+}\, D_{2,+}\, u$, so that it does not involve the corner 
value of $u$.
\item In the derivation of the relation that is analogous to \eqref{equation_B2}, a boundary term appears because we now only sum with respect to $k \in \N$. 
However, this boundary term turns out to be zero. Indeed, for the sequence $W\,:=\, D_{1,+}\,D_{2,+}\, u$, we compute:
\begin{equation}\nonumber
\|\, A_{2,-} W \, \|_{\ell^2(\mathfrak{I})}^2 \, + \, \dfrac{1}{4} \, \|\, D_{2,-} W \, \|_{\ell^2(\mathfrak{I})}^2 
\, = \, \frac{1}{2} \, \sum_{j,k\geq 0} \, W_{j,k}^2 \, + \, \frac{1}{2} \, \sum_{j,k\geq 0} \, W_{j,k-1}^2\, =\,\Vert\, W\Vert_{\ell^2(\mathfrak{I})}^2 \, ,
\end{equation}
because $W_{j,-1}\,=\,0$ for all $j\geq 0$ (from the boundary condition \eqref{extrapolationj}).
\end{itemize}
We can thus adapt and reproduce (almost word for word) the proof of Lemma \ref{lem8} to get \eqref{est-finale-prop32} and \eqref{est-finale-prop31} is obtained 
now in a similar way by exchanging the roles of $j$ and $k$. We then use \eqref{est-finale-prop3} in \eqref{inegprop3-3} and obtain the claim of Proposition \ref{prop3}.
\end{proof}

To complete the proof of Theorem \ref{thm2}, we now combine the results of Lemmas \ref{lem9} and \ref{lem10} with Proposition \ref{prop3}. We recall the energy 
balance:
\begin{align*}
\Vert u^{n+1}\Vert^2_{\ell^2(\mathfrak{I})}\,-\,\Vert u^n\Vert^2_{\ell^2(\mathfrak{I})}\,= \, &\,\Vert w^n\Vert^2_{\ell^2(\mathfrak{I})}\,+\,\Vert v^n\Vert^2_{\ell^2(\mathfrak{I})} 
\,-\,2\,\langle u^n\,;\, w^n\rangle_{\ell^2(\mathfrak{I})}\,-\,\lambda\, \vert \, a\,\vert\, \mu \,\vert \,b\, \vert \, (u_{0,0}^n)^2 \, \\
&+\,2\,\langle u^n\,;\, v^n\rangle_{\ell^2(\mathfrak{I})}\,-2\,\langle v^n\,;\, w^n\rangle_{\ell^2(\mathfrak{I})}\,+\,\lambda\, \vert \, a\,\vert\, \mu \,\vert \,b\, \vert \, 
(u_{0,0}^n)^2 \, .
\end{align*}
Thanks to Lemma \ref{lem10} and Proposition \ref{prop3}, we obtain (quite like for the half-space problem except for the new corner contribution which we have to 
subtract here):
\begin{align}
\nonumber
\Vert w^n\Vert^2_{\ell^2(\mathfrak{I})}\,&+\,\Vert v^n\Vert^2_{\ell^2(\mathfrak{I})}\,-\,2\,\langle u^n\,;\, w^n\rangle_{\ell^2(\mathfrak{I})} 
\,-\,\lambda\, \vert \, a\,\vert\, \mu \,\vert \,b\, \vert \, (u_{0,0}^n)^2 \\
\nonumber \leq \, & \,\frac{(\lambda\, a )^2}{2}\,\Big( (\lambda\, a)^2\,+\,(\mu\, b)^2\,-\,\frac{1}{2}\Big)\,\Vert \Delta_1\, u^n\Vert_{\ell^2(\mathfrak{I})} 
\,+\,\frac{(\mu\, b )^2}{2}\,\Big( (\lambda\, a)^2\,+\,(\mu\, b)^2\,-\,\frac{1}{2}\Big)\,\Vert \Delta_2\, u^n\Vert_{\ell^2(\mathfrak{I})} \\
\label{eq-fin-corner} 
&\,+\,\frac{(\lambda\, a )^2+(\mu\, b)^2}{2}\,\Big( (\lambda\, a)^2\,+\,(\mu\, b)^2\,-\,\frac{1}{2}\Big)\,\Vert D_{1,+}\,D_{2,+}\, u^n\Vert_{\ell^2(\mathfrak{I})} \\
\nonumber 
&\,-\,\frac{(\mu\, b)^2}{8}\, ((\lambda\, a)^2\,+\,(\mu\, b)^2)\,\sum_{k\geq 0}(\Delta_2\, u^n_{0,k})^2\,-\,\frac{(\lambda\, a)^2}{8}\, ((\lambda\, a)^2\,+\,(\mu\, b)^2)\,\sum_{j\geq 0}(\Delta_1\, u^n_{j,0})^2 \, ,
\end{align} 
where the three first terms on the right hand side of \eqref{eq-fin-corner} are negative under the CFL condition \eqref{CFLlaxwendroff}. Assuming from now on 
that \eqref{CFLlaxwendroff} holds, we thus get:
\begin{align}
\nonumber
\Vert u^{n+1}\Vert^2_{\ell^2(\mathfrak{I})}\,-\, & \, \Vert u^n\Vert^2_{\ell^2(\mathfrak{I})}\,\le 
2\,\langle u^n\,;\, v^n\rangle_{\ell^2(\mathfrak{I})}\,-2\,\langle v^n\,;\, w^n\rangle_{\ell^2(\mathfrak{I})}\,+\,\lambda\, \vert \, a\,\vert\, \mu \,\vert \,b\, \vert \, (u_{0,0}^n)^2 \\
&\,-\,\frac{(\mu\, b)^2}{8}\, ((\lambda\, a)^2\,+\,(\mu\, b)^2)\,\sum_{k\geq 0}(\Delta_2\, u^n_{0,k})^2\,-\,
\frac{(\lambda\, a)^2}{8}\, ((\lambda\, a)^2\,+\,(\mu\, b)^2)\,\sum_{j\geq 0}(\Delta_1\, u^n_{j,0})^2 \, ,\label{eq-fin-corner'} 
\end{align}

We now estimate the remaining terms in the energy balance. We begin with looking at the first term on the right hand side of \eqref{lem9-antisym2}. 
Using Young's inequality, we obtain:
\begin{multline*}
\frac{\lambda\, \vert \, a\, \vert \,(\mu\, b)^2}{2} \, \left| \, \sum_{k\geq 0}\,u_{0,k}^n\,\Delta_{2}\, u_{0,k+1}^n \, \right| \\
\leq \, \frac{(\lambda\, a)^2 \, (\mu\, b)^2}{\sqrt{2} \, ((\lambda\, a)^2 + (\mu\, b)^2)} \, \sum_{k\geq 0}\, (u_{0,k}^n)^2 \, + \, 
\frac{(\mu\, b)^2 \, ((\lambda\, a)^2 + (\mu\, b)^2)}{8 \, \sqrt{2}} \, \sum_{k\geq 1}\, (\Delta_2 \, u_{0,k}^n)^2 \, .
\end{multline*}
For the first term on the right hand side, we use the inequality:
$$
\frac{\lambda\, |\, a \, | \, \mu \, | \, b \, |}{(\lambda\, a)^2 + (\mu\, b)^2} \, \le \, \dfrac{1}{2} \, ,
$$
and $\mu \, | \, b \, | \le 1/\sqrt{2}$, which follows from \eqref{CFLlaxwendroff}. We thus end up with:
\begin{equation*}
\frac{\lambda\, \vert \, a\, \vert \,(\mu\, b)^2}{2} \, \left| \, \sum_{k\geq 0}\,u_{0,k}^n\,\Delta_{2}\, u_{0,k+1}^n \, \right| \, \leq \, 
\frac{\lambda\, |\, a \, |}{4} \, \sum_{k\geq 0}\, (u_{0,k}^n)^2 \, + \, 
\frac{(\mu\, b)^2 \, ((\lambda\, a)^2 + (\mu\, b)^2)}{8 \, \sqrt{2}} \, \sum_{k\geq 1}\, (\Delta_2 \, u_{0,k}^n)^2 \, .
\end{equation*}
In an entirely similar way, we obtain:
\begin{equation*}
\frac{\lambda\, \vert \, a\, \vert \,(\mu\, b)^2}{2} \, \left| \, u_{0,0}^n \, \Delta_{2}\, u_{0,0}^n \, \right| \, \leq \, 
\frac{\lambda\, |\, a \, |}{4} \, (u_{0,0}^n)^2 \, + \, \frac{(\mu\, b)^2 \, ((\lambda\, a)^2 + (\mu\, b)^2)}{8 \, \sqrt{2}} \, (\Delta_2 \, u_{0,0}^n)^2 \, .
\end{equation*}
Arguing similarly for the other expressions on the right hand side of \eqref{lem9-antisym2}, we thus obtain the bound:
\begin{multline}
\label{estimyoungfinal}
\left| \, 2\,\langle v^n\,;\, w^n\rangle_{\ell^2(\mathfrak{I})} \, \right| \, \leq \, 
\frac{\lambda\, |\, a \, |}{2} \, \sum_{k\geq 0}\, (u_{0,k}^n)^2 \, + \, \frac{\mu\, |\, b \, |}{2} \, \sum_{j\geq 0}\, (u_{j,0}^n)^2 \\
+ \, \frac{(\mu\, b)^2 \, ((\lambda\, a)^2 + (\mu\, b)^2)}{8 \, \sqrt{2}} \, \sum_{k\geq 0}\, (\Delta_2 \, u_{0,k}^n)^2 \, + \, 
\frac{(\lambda \, a)^2 \, ((\lambda\, a)^2 + (\mu\, b)^2)}{8 \, \sqrt{2}} \, \sum_{j\geq 0}\, (\Delta_1 \, u_{j,0}^n)^2 \, .
\end{multline}
Using \eqref{lem9-antisym1} and \eqref{estimyoungfinal} in \eqref{eq-fin-corner'} as well as the inequality $1-1/\sqrt{2} \ge 1/4$, we obtain:
\begin{align*}
\Vert u^{n+1}\Vert^2_{\ell^2(\mathfrak{I})}\,-\, & \, \Vert u^n\Vert^2_{\ell^2(\mathfrak{I})}\,\le \, 
- \, \frac{\lambda\, |\, a \, |}{2} \, \sum_{k\geq 0}\, (u_{0,k}^n)^2 \, - \, \frac{\mu\, |\, b \, |}{2} \, \sum_{j\geq 0}\, (u_{j,0}^n)^2
\,+\,\lambda\, \vert \, a\,\vert\, \mu \,\vert \,b\, \vert \, (u_{0,0}^n)^2 \\
&\,-\,\frac{(\mu\, b)^2}{32}\, ((\lambda\, a)^2\,+\,(\mu\, b)^2)\,\sum_{k\geq 0}(\Delta_2\, u^n_{0,k})^2\,-\,\frac{(\lambda\, a)^2}{32}\, ((\lambda\, a)^2\,+\,(\mu\, b)^2)\,\sum_{j\geq 0}(\Delta_1\, u^n_{j,0})^2 \, ,
\end{align*} 
The remaining argument is to observe that \eqref{CFLlaxwendroff} implies that we have:
$$
\lambda\, \vert \, a\,\vert\, \mu \,\vert \,b\, \vert \, (u_{0,0}^n)^2 \, \le \, \dfrac{\lambda\, \vert \, a\,\vert \,+\, \mu \,\vert \,b\, \vert}{2 \, \sqrt{2}} \, (u_{0,0}^n)^2 \, .
$$
Using again the inequality $1-1/\sqrt{2} \ge 1/4$, we end up with \eqref{estimthm2}.

\section{Numerical simulations and perspectives}
\label{part-num}

In order to illustrate our stability result of Theorem \ref{thm2}, we implement the numerical scheme \eqref{LW} in the rectangle $[0,3] \times [0,5]$ with 
$500$ points in the first ($x$) direction and $800$ points in the second ($y$) direction. We choose $a=-2$ and $b=-4$. The time step $\Delta t$ is fixed 
in such a way that we have:
$$
(\lambda \, a)^2 \, +\, (\mu \, b)^2 \, = \, \dfrac{1}{4} \, .
$$
On the incoming sides of the rectangle, we implement the homogeneous Dirichlet boundary condition, and we use \eqref{extrapolationk}, \eqref{extrapolationj}, 
\eqref{corner} on the outgoing boundaries of the rectangle. The initial condition is the Gaussian function:
$$
(x,y) \, \longmapsto \, \exp \left( -\, 10 \, \left( x-\dfrac{3}{2} \right)^2 \, - \, 10 \, \left( y-\dfrac{5}{2} \right)^2 \right) \, .
$$
The $\ell^2$ norm of the numerical solution is depicted in Figure \ref{fig:normel2-1}, and we verify numerically that it is decreasing (this is not exactly the situation 
predicted in Theorem \ref{thm2} because of the two additional boundaries of the rectangle but the Dirichlet boundary conditions make the $\ell^2$ norm decrease 
with respect to the quarter-space so it is likely that the stability property of Theorem \ref{thm2} is not affected by incorporating homogeneous Dirichlet conditions 
on the incoming sides.

\begin{figure}[htbp]
\centering
\begin{tabular}{c}
\includegraphics[height=.27\textheight]{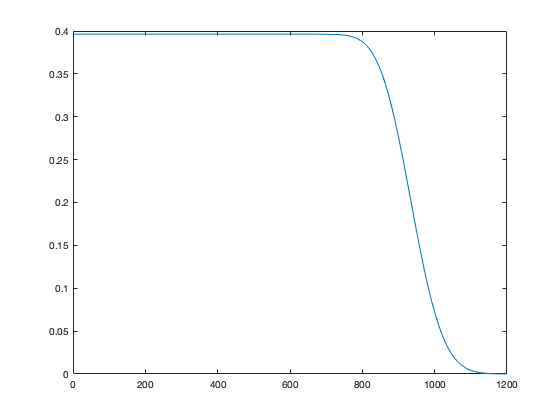} \\
\end{tabular}
\caption{Evolution of the $\ell^2$ norm of the numerical solution with the extrapolation conditions \eqref{extrapolationk}, \eqref{extrapolationj}, \eqref{corner} at 
outgoing boundaries. The norm depends monotonically on the time iteration.}
\label{fig:normel2-1}
\end{figure}

We now run the exact same computation except that we implement on the outgoing corner the condition $u_{-1,-1}^n=290 \, u_{0,0}^n$. The exact value of the 
coefficient is meaningless, it is simply tuned in order to make the illustration visible. The time evolution of the $\ell^2$ norm is depicted in Figure \ref{fig:normel2-2}, 
and we clearly see that the monotonicity property of the $\ell^2$ norm does not hold anylonger. An instability mechanism is taking place, which makes the numerical 
solution become larger and larger in the vicinity of the outgoing corner. Once this mechanism has been ignited, the growth of the $\ell^2$ norm becomes exponential.

\begin{figure}[htbp]
\centering
\begin{tabular}{c}
\includegraphics[height=.27\textheight]{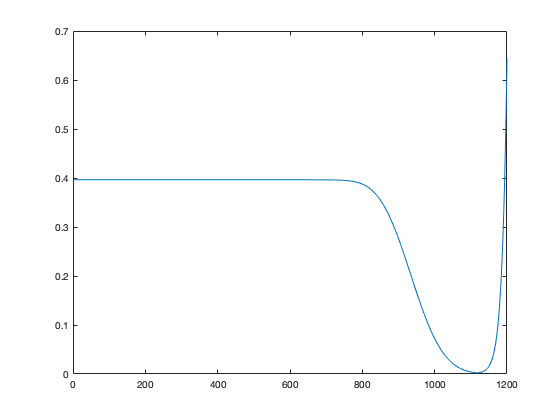} \\
\end{tabular}
\caption{Evolution of the $\ell^2$ norm of the numerical solution with the extrapolation conditions \eqref{extrapolationk}, \eqref{extrapolationj} and $u_{-1,-1}^n=290 \, 
u_{0,0}^n$ at outgoing boundaries. The norm does not depend monotonically on the time iteration anylonger.}
\label{fig:normel2-2}
\end{figure}

Figures \ref{fig:normel2-1} and \ref{fig:normel2-2} are included in order to illustrate that imposing ``good'' corner conditions is crucial in order to maintain stability 
for outgoing transport equations. In the future, we intend to explore the construction of higher order extrapolation conditions as well as extending the theory of 
\cite{Osher-1,Osher-2} to the fully discrete setting in order to be able to analyze the stability of ``general'' boundary conditions in the quarter-space for \eqref{LW}.

\bibliographystyle{alpha}
\bibliography{BC}
\end{document}